\documentclass[hidelinks]{siamart251216}

\usepackage[utf8]{inputenc}
\usepackage{amssymb,bm,euscript,mathrsfs,stmaryrd}
\usepackage{graphicx,array,booktabs,multirow}
\usepackage{enumitem}

\newsiamthm{thm}{Theorem}
\newsiamthm{coro}{Corollary}
\newsiamthm{defi}{Definition}
\newsiamthm{prop}{Proposition}
\newsiamthm{remark}{Remark}

\newcommand{\nc}{\newcommand}
\def\softd{{\leavevmode\setbox1=\hbox{d}%
          \hbox to 1.05\wd1{d\kern-0.4ex{\char039}\hss}}}

\nc{\pd}{ \partial }

\newcommand\dd{\mathrm{d}}
\newcommand\dt{\mathrm{d}t}
\newcommand\dx{\mathrm{d}x}
\newcommand\dS{\mathrm{d}s}

\nc{\pdedge}{ \eth _{\cal D} }
\nc{\pdedgei}{ \eth _{{\cal D}_i} }
\nc{\pdedgej}{ \eth _{{\cal D}_j} }

\newcommand{\frakE}{ \mathfrak{E}}
\newcommand{\frakR}{ \mathfrak{R}}
\newcommand{\RE}{ \mathcal{R}_E}

\newcommand{\Grad}{ \nabla }
\newcommand{\Gradh}{ \nabla }

\newcommand{\Div}{ {\rm div} }

\newcommand\Divh{\Div}

\nc{\shkl}{\sum_{\sigma =K|L\in \faces}}

\nc{\Hc}{ {\cal H} }

\nc{\Dt}{D_t}
\nc{\pdt}{\pd_t}
\nc{\eps}{\varepsilon}

\newcommand{\eu}{e_{\vu}}
\newcommand{\deltau}{\delta_{\vu} }
\newcommand{\etau}{\eta_{\vu} }

\newcommand{\PiR}{\Pi_V}
\newcommand{\PiZ}{\Pi_Z}
\newcommand{\PiQ}{\Pi_Q}

\newcommand{\Pitvu}{\tvu_h}

\newcommand{\norm}[1]{\lVert#1\rVert}
\newcommand{\seminorm}[1]{ \interleave#1 \interleave}

\newcommand\vf{\bm{f}}

\newcommand\vr{\varrho}
\newcommand\vn{\bm{n}}
\newcommand\vu{\bm{u}}
\newcommand\vm{\bm{m}}

\newcommand\vv{\bm{v}}
\newcommand\vw{\bm{w}}

\newcommand\vx{x}

\newcommand{\tvu}{\widetilde{\vu}}
\newcommand{\tp}{\widetilde{p}}

\newcommand{\Td}{{\mathbb T}^d}

\newcommand\vuh{\vu_h}
\newcommand\vvh{\vv_h}

\newcommand{\R}{\mathbb{R}}

\newcommand\grid{\mathcal{T}_h}

\newcommand\faces{\mathcal{E}}

\newcommand\edgesK{\faces(K)}
\newcommand\edgesi{\faces_i}

\newcommand\facesint{\mathcal{E}_{\rm int}}
\newcommand\facesext{\mathcal{E}_{\rm ext}}

\nc{\E}{\mathcal{E}}
\nc{\calD}{\mathcal{D}}
\nc{\Di}{\mathcal{D}_i}
\nc{\Dj}{\mathcal{D}_j}
\nc{\Bii}{\mathcal{B}_{i,i}}
\nc{\Bij}{\mathcal{B}_{i,j}}
\nc{\Bji}{\mathcal{B}_{j,i}}
\nc{\Eij}{ \widetilde{\E}_{i,j}}
\nc{\Eji}{ \widetilde{\E}_{j,i}}

\newcommand{\stiei}{\sum_{\sigma \in \edgesi }}

\newcommand\aleq{\lesssim}

\nc{\gradu}{\nabla \vu}

\nc{\bn}{\vn}
\nc{\bfr}{\bm{r}}
\nc{\bfn}{\vn}
\nc{\bfu}{\vu}
\nc{\bfv}{\vv}
\nc{\bfx}{\vx}
\nc{\avu}{\Ov{\vu}}
\nc{\avuh}{\Ov{\vuh}}
\nc{\vU}{\bm{U}}

\nc{\avU}{\overline{\bm{U}}}

\nc{\abs}[1]{\left\lvert#1 \right\rvert}
\newcommand{\avs}[1]{\left\{\!\!\left\{ #1\right\}\!\!\right\}}

\nc{\jump}[1]{\left\llbracket#1\right\rrbracket}
\newcommand{\jumpB}[1]{%
  [\!\llbracket #1 \rrbracket\!]
}

\newcommand{\order}{ \mathcal O}
\newcommand{\Ov}[1]{\overline{#1}}
\newcommand\vvarphi{{\bm\varphi}}
\newcommand\vvarphih{\vvarphi_h}

\nc{\inner}[1]{( #1 )_{\grid}}
\nc{\innerK}[1]{( #1 )_K} 
\nc{\innerS}[1]{\langle #1 \rangle_\sigma}
\nc{\innerfall}[1]{\langle #1 \rangle_{\pd \grid}}
\nc{\innerfint}[1]{\langle #1 \rangle_{\facesint}}

\nc{\intO}[1]{\int_{\Omega} #1 \dx}
\nc{\intOB}[1]{\int_{\Omega} \left(#1 \right) \dx}
\nc{\intTd}[1]{\int_{\Td} #1  \dx}
\nc{\intTdB}[1]{\int_{\Td} \left(#1 \right) \dx}
\nc{\intTO}[1]{\int_0^\tau \int_{\Td} #1 \dx\dt}
\nc{\intTOB}[1]{\int_0^\tau \int_{\Td}\left( #1 \right) \dx\dt}

\nc{\intDsi}[1]{ \stiei \int_{D_\sigma}#1\dx}
\nc{\intT}[1]{ \int_0^\tau #1 \dt}
\nc{\intTB}[1]{ \int_0^\tau \left(#1 \right)\dt}

\nc{\intS}[1]{\int_{\sigma} #1 \dS}
\nc{\intE}[1]{\int_{\faces} #1 \dS}
\nc{\intSh}[1]{\int_{\sigma} #1 \dS}
\nc{\intK}[1]{\int_{K} #1 \dx}
\nc{\intfacesK}[1]{\int_{\pd K} #1 \dS}
\nc{\intfacesint}[1]{\int_{\facesext} #1 \dS}

\nc{\ceil}[1]{\lceil#1\rceil}
\nc{\floor}[1]{\lfloor#1\rfloor}

\nc{\cmag}{\color{magenta}}
 \nc{\cred}{\color{red}}
 
 \nc{\cblue}{\color{blue}}
 \definecolor{grey}{rgb}{0.5,0.5,0.5}
\nc{\cgrey}{\color{grey}}

\headers{Lax convergence theorems for Euler equations}
{M.\ Luk\'a\v{c}ov\'a-Medvi\v{d}ov\'a and B.\ She}

\title{Lax convergence theorems and improved error estimates of a finite element method
for the incompressible Euler system}

\author{M\'aria Luk\'a\v{c}ov\'a-Medvi\v{d}ov\'a\thanks{%
Institute of Mathematics, Johannes Gutenberg-University Mainz,
Staudingerweg 9, 55 128 Mainz, Germany; RMU Co-Affiliate Technical University
Darmstadt, Germany (\email{lukacova@uni-mainz.de}).
M.L.-M.\ gratefully acknowledges the support of DFG Project 5258 3336
funded within Focused Programme SPP 2410 ``Hyperbolic Balance Laws: Complexity,
Scales and Randomness'' and of the Mainz Institute of Multiscale Modeling.
She is also grateful to the Gutenberg Research College for supporting her research.}
\and Bangwei She\thanks{%
Academy for Multidisciplinary Studies, Capital Normal University,
West 3rd Ring North Road 105, 100048 Beijing, P. R. China
(\email{bangweishe@cnu.edu.cn}).
B.S.\ gratefully acknowledges the support of National Natural Science Foundation
of China, grant No.\ 12571433.}}

\begin{document}

\maketitle

\setlength{\abovedisplayskip}{5pt plus 1pt minus 1pt}
\setlength{\belowdisplayskip}{5pt plus 1pt minus 1pt}
\setlength{\abovedisplayshortskip}{0pt plus 3pt }
\setlength{\belowdisplayshortskip}{3pt plus 1pt minus 1pt}

\begin{abstract}
In this paper, we present convergence theorems for numerical solutions of the incompressible Euler equations. The first result is the Lax--Wendroff-type theorem, while the second can be formulated in the framework of the Lax equivalence theorem.
To illustrate their application, we study a finite element method  based on the $RT_k/P_k$ element pair, $k\geq 0$. 
Applying the concept of relative energy, we derive the convergence rates of the numerical method by two different approaches  and obtain an improved convergence rate $\order(h^{k+1/2})$. Finally, we present numerical experiments that illustrate our theoretical results.
\end{abstract}



\section{Introduction}
The celebrated Lax equivalence theorem \cite{LaxR} stands as one of the cornerstones of numerical analysis, establishing that for linear problems stability and consistency of a numerical scheme are equivalent to its convergence. However, its scope is inherently restricted to a linear setting. In this light, it is noteworthy that in recent works by Feireisl and Luk\'a\v{c}ov\'{a}-Medvid'ov\'a \cite{FLMFCM} and  Feireisl et al. \cite{FLM_dmv,FeLMMiSh}  a nonlinear analogue was developed and applied to compressible fluid flows governed by the Euler or Navier--Stokes system. In these works  a new framework for convergence analysis based on the dissipative weak--strong uniqueness principle has been established. It  extends the spirit of Lax's theorem to a much broader and physically relevant class of nonlinear problems.  This framework has been successfully applied to the convergence analysis of various well-known numerical methods for the compressible Euler and the Navier--Stokes--Fourier equations. We refer to \cite{LMY} for the analysis of the Godunov finite volume scheme, 
\cite{LO} for a high-order discontinuous Galerkin scheme, \cite{ALO} for high-order residual distribution schemes, \cite{Kuzmin} for a finite element flux-correction scheme, as well as to \cite{FLM18_brenner}, where the convergence of  a finite volume diffusive upwind method was proved.
In this context, we also mention the results related to the Lax--Wendroff theorem, which states that a consistent approximation converges to a weak solution under the assumption that it is bounded and converges strongly, see \cite{MJMC,GallouetLW2}.

Despite the fact that the Lax-type theorems for the compressible flows are available, similar results in the context of incompressible Euler equations are still largely absent in the literature, in particular if one aims to have  an unconditional result without any a priori assumptions on the numerical solution. The goal of the present paper is to fill this gap by developing a generalised Lax convergence theory for the incompressible Euler equations.
We will consider general, the so-called ``rough'' initial data, meaning that weak solutions of the incompressible Euler equations may be nonunique,  cf.~\cite{DeLellisSzekelyhidi2010,Shnirelman1997}. This is closely connected to the Onsager conjecture and its resolution: below the critical H\"older regularity threshold $C^{1/3}$ weak solutions may dissipate kinetic energy anomalously and are nonunique, whereas solutions belonging to $C^{\alpha}$, $\alpha>1/3$, conserve energy, cf.~\cite{Buckmaster2019,Isett2018}. This implies that numerical analysis for the incompressible Euler equations is delicate.

In the present paper, we will work in the class of \emph{dissipative weak solutions}, which is larger than the class of weak solutions. Nevertheless, \emph{the dissipative weak--strong uniqueness principle} holds, meaning that all dissipative weak solutions coincide with a strong solution emanating from the same initial data on its lifespan.   This result follows from the application of the relative energy functional that measures the distance of two solutions to the Euler equations. We will also use the relative energy to measure the proximity of a numerical solution $\vuh$ and a strong solution $\tvu$
\[
\RE(\vuh|\tvu) = \intO{\frac12 |\vuh-\tvu|^2}.
\]
To the best of our knowledge, this is the first work that establishes both a Lax--Wendroff-type theorem and a nonlinear Lax equivalence principle for the incompressible Euler system, and applies them to a concrete numerical method. 
The novelty of the paper can be summarized as follows.
\begin{itemize}
\item
First, we develop a generalised Lax convergence framework for the incompressible Euler system. We show that any stable and consistent approximation admits a subsequence converging weakly-(*) to a dissipative weak solution and, under strong convergence, recover a Lax--Wendroff-type result.

\item
Second, combining the above convergence framework with the dissipative weak--strong uniqueness principle, we establish a nonlinear analogue of the Lax equivalence principle: whenever a strong solution exists, every stable and consistent approximation converges strongly to this solution.

\item
Third, we derive an abstract relative-energy estimate that quantifies the convergence of consistent approximations towards a strong solution in terms of their stability and consistency errors.

\item
Finally, we apply the abstract theory to the $RT_k/P_k$ finite element approximation of the incompressible Euler system for all $k\geq0$. We prove unconditional convergence towards a dissipative weak solution and, for smooth solutions, establish the error estimate
\[
    \|\vu_h-\widetilde{\vu}\|_{L^\infty(0,T;L^2(\Omega))}
    \lesssim h^{k+1/2}.
\]
In particular, this extends the analysis of Guzm\'an et al.~\cite{GSS} to the lowest-order case $k=0$, for which the rate $\order(h^{1/2})$ is obtained, and improves their theoretical $\order(h^k)$ estimate by half an order for $k\geq1$.
\end{itemize}

Our framework can be generalised to other nonlinear dynamical systems and structure-preserving numerical methods, e.g., those based on the auxiliary variables approach recently proposed by Andrews and Farrell \cite{BF}. 

The rest of this paper is organised as follows. In Section~\ref{euler}, we introduce the incompressible Euler system and its dissipative weak, weak, and strong solutions, followed by the dissipative weak--strong uniqueness theory. Moreover, we introduce the concept of consistent approximations and discuss their  weak limit. We present the main result on the Lax--Wendroff-type theorem and the generalised Lax equivalence theorem for the incompressible Euler system. 
In Section~\ref{Con_MAC}, we study a finite element scheme with $RT_k/P_k$ elements and show its stability and consistency. Applying the above-mentioned generalised Lax-type theorems, we prove the convergence of the $RT_k/P_k$ finite element scheme. Section~\ref{con_rate} is devoted to the study of the  convergence rates using 
the relative energy approach. Section~\ref{con_num} presents results of numerical simulations that illustrate the behaviour of the scheme and confirm our theoretical results.

\section{The incompressible Euler system}\label{euler}

The motion of inviscid, incompressible fluids is governed by the incompressible Euler equations 
\begin{equation}\label{PDE}
\begin{aligned}
        \partial_t \bm u  + \bm u \cdot \Grad  \bm u + \Grad p = 0, 
          \quad     \Div  \bm u =0,
\end{aligned}
\end{equation}
 where $\vu$ and $p$ are the fluid velocity and pressure, respectively. 
 System \eqref{PDE} is considered in the time--space cylinder $[0,T]\times\Omega$, where $\Omega\subset\R^d$, and is supplemented with the initial data
\begin{equation}\label{IC}
   \vu(0,\cdot) = \vu^0  \mbox{ on } \Omega 
\end{equation}
and boundary conditions. Here, we assume either no-flux boundary conditions
\begin{subequations}\label{BC}
\begin{equation}
\label{BC1}
   \vu\cdot \vn =  0  \mbox{ on } \partial \Omega \times (0,T);
\end{equation}
or periodic boundary conditions, i.e., $\Omega$ can be identified with a flat torus
\begin{equation}
\label{BC2}
\Omega = \Td \equiv \left( [0,L] |_{\{ 0,L \}} \right)^d, {\quad L>0.}
\end{equation}
\end{subequations}
We proceed by defining different solutions of the incompressible Euler equations \eqref{PDE}.

\subsection{Solution concepts}\hspace{-0.2cm}
We denote by $L^2_\sigma(\Omega)$  the $L^2$-closure of smooth divergence-free vector fields satisfying the corresponding boundary condition.
\begin{defi}[Dissipative weak solution]\label{defi_DWs}
Let $\Omega \subset \R^d$, $d=2,3$. We say that  $\vu \in  C_{weak}([0,T];L^2(\Omega;\R^d))$ is a dissipative weak (DW) solution of the incompressible Euler system \eqref{PDE}--\eqref{BC} 
 with the initial data $\vu^0 \in L^2_{\sigma}(\Omega)$ provided the following hold:
\begin{subequations}\label{dws}
\begin{itemize}
    \item {\bf Energy inequality.} There is a defect measure $\frakE \in L^{\infty}(0,T; \mathcal{M}^+({\Omega}))$\footnote{Here $\mathcal{M}^+(\Omega)$ denotes the set of positive Radon measures on $\Omega$.} 
    such that the following energy inequality holds for a.a. $\tau\in[0,T]$
    \begin{equation}\label{dw1}
        \intO{\frac{1}{2} |\vu(\tau, \cdot)|^2  } + \int_{\Omega} d\frakE(\tau) \leq \intO{\frac{1}{2} |\vu^0|^2  };
    \end{equation}
    \item {\bf Divergence free.}    It holds for a.a. $t\in(0,T)$ and $\varphi \in C^M( \Omega)$, $M\geq 1$, that
    \begin{equation}\label{dw2}
       \intO{\vu(t) \cdot \Grad\varphi} =0;
    \end{equation}
    \item {\bf Momentum equation.}
    It holds for any $0 \leq \tau \leq T$ and any divergence-free test function $\vvarphi \in C^M([0,T] \times {\Omega}; \R^d), M \geq 1$, satisfying periodic boundary conditions or $\vvarphi \cdot \vn|_{\partial \Omega} =0$ that
    \begin{equation}\label{dw3}
    \begin{aligned}
    \left[ \intO{\!\!\vu \cdot \vvarphi} \right]^{t=\tau}_{t=0}   
    \!\!= \!\! \intT{\!\!\!\!\intOB{\!\vu \cdot \partial_t\vvarphi + (\vu \otimes \vu): \Grad\vvarphi
    }} 
    +\intT{\!\!\!\!\int_{{\Omega}} \!\! \Grad\vvarphi \! : \!d\frakR(t)}
    \end{aligned}
    \end{equation}
      with the Reynolds defect 
     $\frakR \in L^\infty(0,T;\mathcal{M}^+({\Omega}; \R_{\rm sym}^{d\times d}));$
     \item {\bf Compatibility condition.}
     \begin{equation}\label{dw4}
         \underline{d} \frakE \leq \operatorname{tr}\frakR \leq \Ov{d} \frakE \mbox{ {\em for some constants} } 0 \leq \underline{d} \leq \Ov{d}.
     \end{equation}
\end{itemize}
\end{subequations}
\end{defi}

\begin{remark}
    A dissipative weak solution $\vu$ is an admissible (energy-stable) weak solution if the defect vanishes $\frakE=0$.
\end{remark}
\begin{defi}[Strong solution]\label{defi_Cs}
Let $\Omega=\Td$, or let $\Omega\subset\R^d$, $d=2,3$, be a bounded domain with sufficiently smooth boundary. Let $\vu^0\in H^m(\Omega;\R^d)$, $m>d/2+1$, satisfy $\Div\vu^0=0$ together with the corresponding boundary condition. We say that $(\tvu,\tp)$ is a strong solution of the incompressible Euler system if it satisfies \eqref{PDE}--\eqref{BC},
\[
 \tvu\in C([0,T];H^m(\Omega;\R^d))
 \cap C^1([0,T];H^{m-1}(\Omega;\R^d)),
 \quad
 \tp\in C([0,T];H^m(\Omega)\cap L^2_0(\Omega)).
\]
In addition, it holds 
$ -\Delta\tp=\Div\big(\Div(\tvu\otimes\tvu)\big)$ 
with periodic boundary conditions, or with the corresponding Neumann boundary condition in the no-flux case.
\end{defi}
\begin{remark}
For local well-posedness result on a torus or on a smooth bounded domain, we refer to Ebin and Marsden~\cite{EbinMarsden}; Kato~\cite{Kato} studied the Cauchy problem in $\R^3$. 
For the convergence results in Sections~2 and~3, it is sufficient to take $m\geq3$. Indeed, for $d=2,3$, the Sobolev embeddings yield
\[
    \widetilde{\vu}
    \in C([0,T];W^{1,\infty}(\Omega;\R^d))
    \cap C^1([0,T];C(\overline\Omega;\R^d)).
\]
This is the regularity required in the weak--strong uniqueness
principle and in the abstract relative-energy estimate.

For the error estimates in Section~4, the regularity depends on
the polynomial degree $k$. It is sufficient to assume 
$m\geq k+3$. 
Then
\[
    \widetilde{\vu}
    \in
    H^1(0,T;H^{k+1}(\Omega;\R^d))
    \cap
    L^\infty(0,T;W^{1,\infty}(\Omega;\R^d))
    \cap
    L^2(0,T;W^{k+1,\infty}(\Omega;\R^d)),
\]
which provides the regularity required in Theorems~\ref{thm_r1}
and~\ref{thm_r2}.
\end{remark}

Note that a DW solution coincides with the strong solution as long as the latter exists, see the proof in Appendix \ref{ap_ws} and also a similar result of Wiedemann \cite{Wiede}. 
\begin{lemma}
[Dissipative weak--strong uniqueness]\label{lem_wsuni}
Let $\Omega=\Td$, or let $\Omega\subset \R^d$, $d=2,3$, be a bounded domain with sufficiently smooth boundary. Suppose that $\vu$ is a dissipative weak (DW) solution of the incompressible Euler system in the sense of Definition~\ref{defi_DWs}.
 Let $\tvu$ be a strong solution of the same system and the same initial data in the sense of Definition~\ref{defi_Cs} with $m\geq 3$. 
Then 
$$ \vu=\tvu\quad \mbox{in }(0,T)\times\Omega,\qquad \frakE=\frakR=0.$$
\end{lemma}

\subsection{Consistent approximations}
\begin{defi}[Consistent approximations]\label{defi_CA}
Let $\vu^0\in  L^2_\sigma(\Omega)$. We say that a sequence $\{\vuh\}_{h\searrow0}$, $\vuh \in L^\infty(0,T; L^2(\Omega;\R^d)
)$, 
is a consistent approximation of the incompressible Euler system \eqref{PDE}--\eqref{BC} if
\[
 \vuh(0)\longrightarrow\vu^0\quad\mbox{weakly in }L^2(\Omega;\R^d)
\]
and  the following stability and consistency conditions hold. \\
\begin{subequations}\label{CA}
\noindent{\bf i) The stability condition.}
\begin{equation}\label{CA1}
 \intO{|\vuh(t,\cdot)|^2}\leq\intO{|\vu^0|^2}  + e_E(h) 
\end{equation}
 for a.a. $t \in (0,T)$, where $|e_E(h)| \to 0$ as $h \to 0$;\\
\noindent{\bf ii) The consistency conditions.}
\begin{itemize}
\item {\bf Divergence free.} It holds for any  $\psi\in C^M([0,T]\times \overline{\Omega})$, $M\geq1$, that
\begin{equation}\label{CA2}
 \intO{\vuh(t)\cdot\Grad\psi(t)}=e_{\vr}(t,h,\psi)
 \quad\mbox{for a.a. }t\in(0,T),
\end{equation}
where 
$
 \int_0^T|e_{\vr}(t,h,\psi)|\,\dt\longrightarrow0
 \quad\mbox{as }h\to0.
$
\item {\bf Momentum equation.} Let $\vvarphi\in C^M([0,T]\times\Omega;\R^d)$, $M\geq1$, be divergence-free and satisfy the periodic boundary condition or $\vvarphi\cdot\vn|_{\partial\Omega}=0$. Then, for any $\tau\in[0,T]$,
\begin{equation}\label{CA3}
 \left[\intO{\vuh\cdot\vvarphi}\right]_0^\tau
 =\intT{\intOB{\vuh\cdot\partial_t\vvarphi+(\vuh\otimes\vuh):\Grad\vvarphi}}
 +e_{\vm}(\tau,h,\vvarphi),
\end{equation}
where 
$
|e_{\vm}(\tau,h,\vvarphi)|\longrightarrow0
 \quad\mbox{as }h\to0.
$
\end{itemize}
\end{subequations}
\end{defi}

\begin{remark}
Note that stability condition \eqref{CA1} expresses that the kinetic energy is either conserved or may dissipate over time, but its value is bounded by the initial energy. For the purposes of the present study, it is sufficient to control only the kinetic energy. However, one may be interested in controlling other structure-preserving variables, e.g.,~the helicity as proposed in \cite{R}.    
\end{remark}

\subsection{Convergence}\label{GL}

In this section, we show the existence of a DW solution that can be obtained as a weak limit of a sequence of consistent approximations.  We derive the Lax--Wendroff-type theorem for consistent approximations and also prove the generalised Lax equivalence theorem for the incompressible Euler equations.  Finally, we present a result on the abstract error estimates for the consistent approximations of the incompressible Euler system. These results will provide a theoretical framework for the convergence and error analysis of the $RT_k/P_k$ finite element scheme presented in Sections~\ref{Con_MAC} and~\ref{con_rate}.

\begin{lemma}[Unconditional limit of consistent approximations]\label{lem_con_dw}
Let $\{\vuh\}_{h \searrow0}$ be a consistent approximation of the incompressible Euler system \eqref{PDE}--\eqref{BC} in the sense of Definition \ref{defi_CA}. 
    Then there exists a subsequence of $\vuh$ (not relabelled), such that
\begin{equation}\label{N1}
\vu_h \rightarrow \vu \mbox{ weakly-(*) in }  L^\infty(0,T; L^2(\Omega;R^d)),
\end{equation}
where $\vu$ is a DW solution 
in the sense of Definition \ref{defi_DWs}. 
\end{lemma}
\begin{proof}
The convergence in linear terms with respect to $\vuh$ in \eqref{CA2}
and \eqref{CA3} is straightforward. Weak convergence in the convective
term yields
\[
\vuh\otimes\vuh\to\overline{\vu\otimes\vu}
\quad\mbox{weakly-(*) in }
L^\infty(0,T;\mathcal M^+(\Omega;\R^{d\times d}_{\rm sym})).
\]
Consequently, the Reynolds defect is defined by
\[
\frakR=\overline{\vu\otimes\vu}-\vu\otimes\vu .
\]
Similarly, the weak convergence of the energy term yields the energy
defect
\[
\frakE=
\overline{\frac12|\vu|^2}-\frac12|\vu|^2
\in L^\infty(0,T;\mathcal M^+(\Omega)).
\]
Finally, applying \cite[Proposition 5.3]{FeLMMiSh} with
$G(\vuh(t))=\intO{\frac12|\vuh|^2}$ and
$F(\vuh(t))=\intO{\vuh\otimes\vuh}$ for a.a. $t\in(0,T)$,
we obtain the compatibility condition \eqref{dw4}. 

It remains to identify a weakly continuous representative of the
limit. Testing the limiting momentum relation by
$\eta(t)\varphi(x)$, with $\eta\in C_c^1(0,T)$ and $\varphi$ smooth,
divergence-free and satisfying the corresponding boundary condition,
we obtain
\[
    \frac{\mathrm d}{\mathrm dt}(\vu,\varphi)\in L^\infty(0,T).
\]
Hence $t\mapsto(\vu(t),\varphi)$ is a continuous mapping. 
Since smooth solenoidal vector fields are dense in
$L_\sigma^2(\Omega)$, the uniform $L^2$ bound allows us to extend
this continuity to every $\varphi\in L_\sigma^2(\Omega)$. Hence,
\[
    \vu\in C_{\rm weak}([0,T];L_\sigma^2(\Omega)).
\]
Finally, a standard time cut-off argument together with the $L^2$-weak convergence of the initial data, 
we obtain $\vu(0)=\vu^0$. 
This finishes the proof. 
\end{proof}

Having a sequence of consistent approximations, the following two results, the Lax--Wendroff-type theorem and the generalised Lax equivalence theorem, will be derived. 

\begin{thm}[Lax--Wendroff theorem]\label{thm_LW}
Let $\{\vu_h\}_{h\searrow 0}$  satisfy the consistency formulation of the divergence-free and momentum equations, \eqref{CA2} and \eqref{CA3}. 
If $\vuh$ is bounded in the sense of \eqref{CA1} and converges strongly, i.e., $\norm{\vuh -\vu}_{L^\infty(0,T;L^2(\Omega;\R^d))}$ $\to 0$ as $h\to0$.  Then the limit $\vu$ is a weak solution.
\end{thm}
\begin{proof}
As $\vuh$ converges strongly, passing to the limit $h \to 0$ in \eqref{CA1}--\eqref{CA3},  we obtain  \eqref{dw1}--\eqref{dw3} with zero defects, which completes the proof.   
\end{proof}

We are now ready to prove a generalised version of the Lax equivalence theorem for the incompressible Euler system.
\begin{thm}[A generalised Lax equivalence theorem]\label{thm_LE}
Let $\{\vu_h\}_{h\searrow0}$ satisfy the consistency formulation of the divergence-free and momentum equations, \eqref{CA2} and \eqref{CA3}.
Assume that the Euler system \eqref{PDE}--\eqref{BC} admits a strong solution $\tvu$ in the sense of Definition~\ref{defi_Cs}. Then
\begin{itemize}
\item[(i)] If $\vuh$ is stable in the sense of \eqref{CA1}, then
\begin{equation}\label{conv}
 \vuh\longrightarrow\tvu
 \quad\mbox{strongly in }L^q(0,T;L^2(\Omega;\R^d)),
 \qquad 1\leq q<\infty.
\end{equation}
\item[(ii)] If $\vuh$ converges in the sense of \eqref{conv}, then it is stable in $L^q(0,T;L^2(\Omega;\R^d))$.
\end{itemize}
\end{thm}
\begin{proof}
The second item is obvious. The proof of the first item is more involved. If a sequence $\{\vu_h\}_{h\searrow0}$ satisfies both the consistency formulation and the stability condition as stated in item~(i), then it is a consistent approximation according to Definition~\ref{defi_CA}. Thus, it converges to a DW solution; see Lemma~\ref{lem_con_dw}. Furthermore, since a DW solution coincides with the strong solution; see Lemma~\ref{lem_wsuni}, we conclude that the sequence converges strongly to the strong solution. This completes the proof.
\end{proof}
\begin{remark}
Let us briefly summarise different convergence scenarios.
\begin{itemize}[leftmargin=2.5em,labelsep=0.5em]
\item The convergence of a consistent approximation towards a DW solution is unconditional.
\item If the above convergence is strong, then the limit is either a weak solution or a strong solution.
\item If a strong solution exists, then any consistent approximation converges strongly to this solution.
\end{itemize}
\end{remark}

We proceed by presenting a framework of the abstract  error estimates based on the relative energy inequality.

\begin{thm}[Abstract error estimates]\label{thm_re}
Let $\{\vuh\}_{h\searrow 0}$ be a consistent approximation of the incompressible Euler system \eqref{PDE}--\eqref{BC} in the sense of Definition \ref{defi_CA}. 
Let $(\tvu,\tp)$ be a strong solution of the same problem in the sense of Definition \ref{defi_Cs}. Then 
\begin{multline*}
\norm{\vuh -\tvu}_{L^\infty(0,T;L^2(\Omega;\R^d))}^2 = 2 \operatorname*{ess\,sup}_{t\in(0,T)} \RE(\vuh|\tvu)(t) 
    \\  \leq e^{c_1 T} \bigg( c_2 \bigg( \sup_{0\leq\tau\leq T}|e_{\vm}(\tau,h,\tvu)|
    +\int_0^T\left|e_{\vr}\left(t,h,\frac{|\tvu|^2}{2}-\tp\right)\right|\,\dt\bigg) + \RE(\vuh|\tvu)(0)\bigg), 
\end{multline*}
where the constants $c_1$ and $c_2$ depend on $\norm{\tvu}_{L^\infty(0,T; W^{1,\infty}(\Omega;\R^d))}$, and $e_{\vr}$ and $e_{\vm}$ are the consistency errors of the divergence-free and momentum equations, see Definition \ref{defi_CA}.
\end{thm}
\begin{proof} See Appendix~\ref{ap_re}.
\end{proof}

\section{A finite element method}\label{Con_MAC}
In this section, we introduce the $RT_k/P_k$ finite element scheme originally studied by Guzm\'an et al. \cite{GSS}  for the incompressible Euler system. Our aim is to apply the Lax-type theorems built in the previous section to the convergence analysis of this scheme. 

\subsection*{Notations} 
Let $\grid$ be a shape-regular and quasi-uniform simplicial triangulation of $\Omega$, and let $h=\max_{K\in\grid}h_K$ be the mesh size. We denote by $\faces$ the set of all faces and by $\facesint$ the set of all interior faces. In the periodic case, periodically identified boundary faces are treated as interior faces. For every $\sigma=K^+\cap K^-\in\facesint$, we fix a unit normal $\vn_\sigma$ pointing from $K^+$ to $K^-$. The traces from $K^+$ and $K^-$ are denoted by $v^+$ and $v^-$, respectively, and
\[
 \avs{v}=\frac12(v^++v^-),\qquad
 \jumpB{v}=v^+-v^-.
\]
We further set
\[
 \inner{u,v}=\sum_{K\in\grid}\int_K u\cdot v\,\dx,
 \qquad
 \innerfint{u,v}=\sum_{\sigma\in\facesint}\int_\sigma u\cdot v\,\dS,
\]
\begin{equation*}
 b(\vu,\vv;\vw)
 :=\inner{\vu\cdot\Gradh\vv,\vw}
 -\innerfint{(\vu\cdot\vn)\jumpB{\vv},\avs{\vw}}
 +\innerfint{|\vu\cdot\vn|\jumpB{\vv},\jumpB{\vw}}.
\end{equation*}
Given $\vu\in H_{\rm div}(\Omega)$ with $\Div\vu=0$ and satisfying the periodic or no-flux boundary condition, we have for any $\vv\in H^1(\grid)$ and $\vw\in H^1(\grid)$ that
\begin{equation}\label{udu1}
   \inner{\vv\otimes \vu,\Grad \vw} + \inner{\vu \cdot \Grad\vv,\vw} 
   =\innerfall{(\vu\cdot\vn) \vv, \vw}
   =\innerfint{\vu\cdot\vn,\jumpB{ \vv \cdot \vw}},
\end{equation}
\begin{equation}\label{udu2}
   \inner{\vu \cdot \Grad\vv,\vv} = \frac12\inner{\vu,\Grad|\vv|^2} 
   =\frac12 \innerfall{\vu\cdot\vn,|\vv|^2 }
   =\innerfint{(\vu \cdot\vn) \jumpB{\vv}, \avs{\vv}}
\end{equation}
\begin{equation}\label{udu}
  b(\vu,\vv;\vv) = \innerfint{|\vu \cdot \vn|\jumpB{\vv},\jumpB{\vv}} =: \seminorm{\vv}_{\vu}^2.
\end{equation}

 \subsection*{Function spaces} 
Let $P_k^m$ be the space of polynomials of degree not greater than $k$ for $m$-dimensional vector-valued functions ($m = 1$ for scalar functions). We denote by  $V_h$ the Raviart--Thomas finite element space of order $k$, and  by $Q_h$ the discontinuous piecewise polynomial space of degree at most $k$ over the mesh $\grid$:
\begin{align*}
 V_h&=\left\{\vvh\in H_{\rm div}(\Omega):
       \vvh|_K\in RT_k(K)\ \forall K\in\grid,
       \quad \vvh\cdot\vn=0\text{ on }\partial\Omega\right\},\\
 \widehat Q_h&=\left\{q_h\in L^2(\Omega):q_h|_K\in P_k(K)
       \ \forall K\in\grid\right\},\qquad
 Q_h=\widehat Q_h\cap L^2_0(\Omega),
\end{align*}
where $RT_k(K)=P_k^d(K)+ {\bm x}P_k(K)$.

For periodic boundary conditions, $V_h$ is replaced by the corresponding periodic Raviart--Thomas space.  
Further, we define the exactly divergence-free subspace
\[
Z_h = \{ \vv \in V_h: \; \Divh \vv =0 \}.
\]
This subspace satisfies the following polynomial structure, see \cite[Lemma 3.1]{Duran2008} or Appendix \ref{ap4} for its proof. 
\begin{lemma}[Polynomial structure of $Z_h$]
\label{lem_Zh}
If $\vvh\in Z_h$, then
\[
 \vvh|_K\in P_k^d (K)
 \qquad\forall K\in\grid.
\]
Consequently, $\Gradh\vvh|_K\in P_{k-1}^{d\times d}(K)$, where $P_{-1}(K)=\{0\}$. In particular, when $k=0$, every function in $Z_h$ is elementwise constant.
\end{lemma}

In the following, we recall some standard properties of the $RT_k$ finite element space, see, e.g., Ern and Guermond~\cite{EG,EG2021} for more details. We denote by  $\PiR$ the  canonical Raviart--Thomas interpolation operator for $K\in\grid$ 
\begin{subequations}\label{rt_dofs}
\begin{align}
 \int_\sigma(\PiR\vv-\vv)\cdot\vn_K\,q\,\dS&=0
 &&\forall q\in  P_k(\sigma),\quad
   \sigma \in \edgesK ,\label{rt_dofs_face}\\
 \int_K(\PiR\vv-\vv)\cdot\bm r\,\dx&=0
 &&\forall\bm r\in P_{k-1}^d(K),\qquad k\geq1.
 \label{rt_dofs_cell}
\end{align}
\end{subequations}
The interpolant satisfies the commuting identity $\Div(\PiR\vv)=\PiQ(\Div\vv)$, where $\PiQ$ is the elementwise $L^2$-projection onto $\widehat Q_h$.  In particular, it implies that
\begin{equation}\label{commute}
\Div(\PiR \vv) =0 \quad \mbox{as soon as} \quad \Div \vv =0.
\end{equation}
Moreover,  for $1\leq m\leq k+1$ and $p \geq 2d/(d+2)$, we have the interpolation estimates
\begin{equation}\label{proj}
 \norm{\vv-\PiR\vv}_{L^p(K)}
 +h_K\norm{\Grad(\vv-\PiR\vv)}_{L^p(K)}
 \aleq h_K^m |\vv|_{W^{m,p}(K)},
\end{equation}
where we have used the abbreviations $\norm{\cdot}_{L^p}$, $\norm{\cdot}_{W^{q,p}}$, $\norm{\cdot}_{L^p L^q}$, and $\norm{\cdot}_{L^p W^{q,q}}$ for $\norm{\cdot}_{L^p(\Omega)}$, $\norm{\cdot}_{W^{q,p}(\Omega)}$, $\norm{\cdot}_{L^p(0,T;L^q(\Omega))}$, and $\norm{\cdot}_{L^p(0,T;W^{q,q}(\Omega))}$, respectively. Moreover, the notation $a\aleq b$ means $a\leq cb$ for some positive constant $c$ independent of the mesh size $h$.
\subsection*{Projection of the initial data}
For $\vu^0\in L^2_\sigma(\Omega)$, let $\PiZ \vu^0\in Z_h$ denote the $L^2$-orthogonal projection onto $Z_h$. It satisfies
\begin{equation}\label{piz-sta}
 \norm{\PiZ \vu^0}_{L^2}\leq\norm{\vu^0}_{L^2},
 \qquad
 \PiZ \vu^0\longrightarrow\vu^0\quad\mbox{strongly in }L^2(\Omega;\R^d).
\end{equation}
 Moreover, for every sufficiently smooth divergence-free $\vv$ satisfying the boundary condition, there holds
\begin{equation}\label{piz}
 \norm{\vv-\PiZ \vv}_{L^2}
 = \inf_{\vw \in Z_h}  \norm{\vv- \vw}_{L^2}
 \leq\norm{\vv-\PiR\vv}_{L^2}
 \aleq h^{k+1}\norm{\vv}_{H^{k+1}}.
\end{equation}
 \subsection*{Finite element method}
We introduce a semi-discrete finite element approximation of the incompressible Euler system: 
Let $\vuh(0):={\PiZ \vu^0}$ be the discrete initial data. Find $(\vuh,p_h)(\tau)\in V_h\times Q_h$ such that the following equations hold for any $(\vvarphih,\psi_h)\in V_h\times Q_h$
\begin{subequations}\label{skm}
\begin{equation}\label{skm1}
   \inner{ \pdt \vuh , \vvarphih } + b(\vuh,\vuh;\vvarphih)- \inner{ p_h, \Divh \vvarphih   } =0,
\end{equation}
\begin{equation}\label{skm2}
    \inner{\psi_h, \Divh \vuh } = 0.
\end{equation}
\end{subequations}

It is easy to show that the divergence-free property is preserved by the $RT_k/P_k$ method \eqref{skm}. In fact, setting $\psi_h=\Divh\vuh\in Q_h$  yields $\intO{|\Divh \vuh|^2} =0$, which implies 
    \begin{equation}\label{divfree}
        \Divh \vuh  \equiv 0 .
    \end{equation}
Moreover, for any $\vvh\in Z_h$ and $\psi\in H^1(\Omega)$ we have
\begin{equation}\label{weakdiv}
 \inner{\vvh, \Grad\psi}  
    = \innerfall{\vvh \cdot \vn, \psi} - \inner{  \psi, \Divh \vvh    }
    = \innerfint{\jump{\vvh\cdot\vn},\psi}=0. 
    \end{equation}
  Next, we show the stability of the $RT_k/P_k$ method \eqref{skm}. 
\begin{lemma}[Stability]\label{lm_sta}
    Let $(\vuh,p_h)$ be a solution to the $RT_k/P_k$ method \eqref{skm} with $h \in(0,1)$. Then for any $\tau\in(0,T)$ the following energy stability holds
    \begin{equation}\label{sta}
       \intO{\frac12 |\vuh(\tau,\cdot)|^2} 
       +  \int_0^\tau  \innerfint{|\vuh \cdot \vn|, |\jumpB{\vuh}|^2} \dt
       =   \intO{\frac12 |\vuh^0|^2} \leq \intO{\frac12 |\vu^0|^2}  .
    \end{equation}
Consequently, 
\begin{equation}\label{sta-2}
\norm{\vuh}_{L^\infty(0,T;L^2(\Omega;\R^d))}\lesssim 1 ,\qquad
\int_0^T\innerfint{|\vuh\cdot\vn|,|\jumpB{\vuh}|^2}\,\dt
\lesssim 1 .
\end{equation}
\end{lemma}
\begin{proof}
Summing up \eqref{skm1} and \eqref{skm2} with the test functions $(\vvarphih,\psi_h) = (\vuh, p_h)$ and using the equality \eqref{udu}  we obtain 
\begin{equation*}
 \intO{\pdt |\vuh|^2/2 } +   \innerfint{|\vuh \cdot \vn|, |\jumpB{\vuh}|^2} = 0.
\end{equation*}
Integrating the above equality from time $0$ to $\tau$ 
and using the $L^2$-stability  \eqref{piz-sta} of the projection operator we obtain the energy stability equation \eqref{sta}. 
Consequently, \eqref{sta-2} immediately follow. 
\end{proof}
Following Picard--Lindel\"of theorem and the energy stability equation \eqref{sta}, we obtain a unique solution to the $RT_k/P_k$ method \eqref{skm} that is global in time, see Appendix~\ref{app_exi}. 
\begin{lemma}[Existence of a numerical solution]\label{lem_exi}
Let $\vu^0 \in L^2_{\sigma}$. There exists a unique solution $(\vuh,p_h)\in V_h\times Q_h$ to the $RT_k/P_k$ method \eqref{skm}.
\end{lemma}
With the a priori estimates \eqref{sta-2} we are now ready to show the consistency of the $RT_k/P_k$ method \eqref{skm}.
\begin{lemma}[Consistency]\label{lm_cons}
Let $(\vuh,p_h)$ be a solution to the $RT_k/P_k$ method \eqref{skm} with $h \in (0,1)$, and let $k \geq 0$ be a fixed integer. 
Let $\vvarphi\in L^2(0,T;W^{k+1,\infty}(\Omega;\R^d))$ $\cap W^{1,1}(0,T;H^{k+1}(\Omega;\R^d))$ satisfy $\Div\vvarphi=0$ and the periodic boundary condition or $\vvarphi\cdot\vn=0$ on $\partial\Omega$. Let  $\psi\in L^2(0,T;H^1(\Omega))$. Define the consistency errors
\begin{equation}\label{cs1}
e_{\vm} (\tau ,h, \vvarphi) = \left[ \inner{ \vuh, \vvarphi }  \right]_0^\tau
 - \int_0^\tau \big(\inner{\vuh, \partial_t \vvarphi} + \inner{ \vuh \otimes \vuh, \Grad \vvarphi } \big) \dt, \quad
\tau \in  [0,T].
\end{equation}
\begin{equation}\label{cs2}
    e_{\vr}(t,h,\psi)=\inner{\vuh(t),\Grad\psi(t)}.
\end{equation}
Then, $e_{\vr} (t, h, \psi) = 0 \mbox{ for a.a. } t\in  (0,T)  $ and
 \begin{equation}\label{rateA}
 \sup_{0\leq \tau \leq T}\abs{e_{\vm} (\tau ,h, \vvarphi) }  \aleq h^A,    \;  A =
 \begin{cases}
1/2, & k=0,\\[1mm]
 k, & k \geq 1. 
\end{cases} 
\end{equation}
\end{lemma}
\begin{proof}
{Step 1}: consistency error of the momentum equation. 
Setting  $\vvarphih = \PiR \vvarphi \in V_h$ in \eqref{skm1}  and using the equality \eqref{commute} we have 
\begin{equation*}
\inner{  \pdt \vuh, \PiR \vvarphi} + b(\vuh,\vuh; \PiR\vvarphi ) = \inner{ p_h , \Div \PiR \vvarphi } =0.
\end{equation*}
Subtracting the above equality from the consistency formulation \eqref{cs1}  
we derive $e_{\vm} = \sum_{i=1}^3 R_i $, where 
\begin{align*}
 R_1 &= \left[ \inner{ \vuh , \vvarphi }  \right]_0^\tau 
 - \int_0^\tau \big(\inner{\vuh , \partial_t \vvarphi} + \inner{\pdt \vuh , \PiR \vvarphi} \big)\dt, 
 \\
 R_2 &=  -\int_0^\tau \innerfint{|\vuh \cdot \vn|\jumpB{\vuh},\jumpB{\PiR\vvarphi}} \dt 
=\int_0^\tau \innerfint{|\vuh \cdot \vn|\jumpB{\vuh},\jumpB{\vvarphi - \PiR \vvarphi}} \dt, 
 \\
 R_3 &= -\int_0^\tau \big[\inner{ \vuh \otimes \vuh, \Grad \vvarphi} + \inner{
 \vuh \cdot \Grad \vuh, \PiR \vvarphi} - \innerfint{(\vuh \cdot \vn)\jumpB{\vuh},\avs{\PiR \vvarphi}} \big]\dt 
 \\&=
 \int_0^\tau \inner{ \vuh \cdot \Gradh \vuh,  \vvarphi  -  \PiR \vvarphi}\dt  + \int_0^\tau \innerfint{(\vuh \cdot \vn)\jumpB{\vuh},\avs{\PiR \vvarphi}-\vvarphi} \dt ,
 \end{align*}
where we have used the identity $\jumpB{\vvarphi}=0$ in $R_2$ and \eqref{udu1} in the last equality of $R_3$. Next, using the projection estimates \eqref{proj} and H\"older's inequality we have 
\begin{align*}
|R_1| &= \left| \left[ \inner{ \vuh,  \vvarphi - \PiR \vvarphi }  \right]_0^\tau - \int_0^\tau \inner{\vuh, \partial_t (\vvarphi - \PiR \vvarphi)}\dt \right|
\\ & \leq  h^{k+1} \norm{\vuh}_{L^\infty L^2}\norm{ \vvarphi}_{W^{1,1} H^{k+1}}
\aleq h^{k+1} .
\end{align*}
Next, using H\"older's inequality, the inverse trace estimate, and the a priori estimates \eqref{sta-2} we have
\begin{equation*}
\begin{aligned}
&\left| \int_0^\tau \innerfint{(\vuh \cdot \vn)\jumpB{\vuh}, \vvh} \dt \right|
 \aleq \left( \int_0^\tau \seminorm{\vuh}_{\vuh}^2  \dt \right)^{1/2}  \left( \int_0^\tau \innerfint{|\vuh \cdot \vn|, |\vvh|^2} \dt \right)^{1/2} 
\\&  \aleq  \left(h^{-1}\norm{\vuh}_{L^\infty L^2} \norm{\vvh}_{L^2L^\infty}^2\right)^{1/2}
\aleq h^{-1/2}\norm{\vvh}_{L^2L^\infty}.
\end{aligned}
\end{equation*}
With the above estimates, we have 
\begin{align*}
 |R_2| \aleq 
   h^{-1/2} \norm{\vvarphi - \PiR \vvarphi}_{L^2 L^{\infty}} 
   \aleq      h^{k+1/2} \norm{\vvarphi}_{L^2 W^{k+1,\infty}} \aleq  h^{k+1/2}.
\end{align*}
Analogously, we have
\begin{align*}
|R_3|  &=\left|   \int_0^\tau \inner{ \vuh \cdot \Gradh \vuh,  \vvarphi - \PiR \vvarphi}\dt  - \int_0^\tau \innerfint{(\vuh \cdot \vn)\jumpB{\vuh},\avs{\vvarphi - \PiR \vvarphi}} \dt \right|
\\ &
\leq h^{k+1} \norm{\vuh}_{L^\infty L^2} \norm{\Gradh \vuh}_{L^2 L^2} 
\norm{\vvarphi}_{L^2 W^{k+1,\infty}} 
+ h^{k+1/2}  \norm{\vvarphi }_{L^2 W^{k+1,\infty} } .
 \end{align*}
For the case of $k=0$ we know from Lemma~\ref{lem_Zh} that $\Gradh \vuh =0$, and thus $|R_3|  \aleq  h^{1/2}$. For the case of $k \geq 1$ we have 
\begin{align*}
\norm{\Gradh\vuh}_{L^2(0,T;L^2)}
 \aleq h^{-1}\norm{\vuh}_{L^2(0,T;L^2)}
 \aleq  \norm{\vuh}_{L^\infty(0,T;L^2)} \aleq h^{-1} 
\end{align*}
which implies $|R_3| \aleq h^k$ and completes the estimates of $e_{\vm}$. 

{Step 2:} consistency of the divergence-free constraint. Recalling \eqref{weakdiv} we find  
\begin{align*}
    e_{\vr}(t,h,\psi)=\inner{\vuh(t),\Grad\psi(t)}=0\mbox{ as }\vuh \in Z_h. 
\end{align*}
Consequently, we finish the proof by collecting the estimates of the above two steps. 
\end{proof}

Lemmas~\ref{lm_sta} and \ref{lm_cons} imply that a sequence of finite element solutions to \eqref{skm} forms a consistent approximation of the incompressible Euler equations \eqref{PDE}--\eqref{BC}.  This fact directly  yields the following convergence result.

\begin{prop}[Convergence]\label{Prop1}
Let $\{(\vuh,p_h)\}_{h \searrow 0}$ be a sequence of numerical solutions obtained by the $RT_k/P_k$ method \eqref{skm} with the initial data $\vu^0\in L^2_{\sigma}(\Omega)$. Then we have the following convergence results:
\begin{itemize}
    \item $\vuh \to \vu$ weakly$-(*)$  to a DW solution of the Euler system \eqref{PDE}--\eqref{BC1}. 
    \item  If the convergence is strong, then the limit $\vu$ is an admissible weak solution.
    \item If the Euler system admits a strong solution in the sense of Definition~\ref{defi_Cs}, then $\vuh$ converges strongly to the strong solution. 
\end{itemize}
\end{prop}
\begin{proof}
Recalling  Lemmas \ref{lm_sta} and \ref{lm_cons} we know that $\{\vuh\}_{h\searrow 0}$ is a consistent approximation. Then the proof follows directly from Lemma \ref{lem_con_dw} and Theorems \ref{thm_LW} and \ref{thm_LE}.
\end{proof}
\begin{remark}
 A direct consequence of Proposition~\ref{Prop1} is  the existence of  a DW solution. 
\end{remark}
\section{Error estimates}\label{con_rate}
In this section, we study the convergence rate of numerical solutions to the $RT_k/P_k$ method \eqref{skm}.  More precisely, 
we study the error between a numerical solution $\vuh$  
and the strong solution $\tvu$.  
\subsection{Convergence rate -- I}
First, we show the convergence rate by directly recalling the abstract error estimates (Theorem \ref{thm_re}) with the stability (Lemma \ref{lm_sta}) and consistency (Lemma \ref{lm_cons}) of the $RT_k/P_k$ method \eqref{skm}.  We point out that this approach has been used in our previous works where the convergence rates of finite volume methods were studied for compressible Euler and Navier--Stokes(--Fourier) systems, see~\cite{BLMSY,IEE_FLS,LSY_JSC,LSY_penalty_NS,LSY_penalty_NSF}.
\begin{thm}[Convergence rate -- I]\label{thm_r1}
 Let $\vuh$ be a numerical solution of the $RT_k/P_k$ method \eqref{skm}. 
Let the Euler system admit a strong solution $\tvu$ in the sense of Definition~\ref{defi_Cs} with $m\geq k+3$, which satisfies the regularity class
\[\tvu \in L^2(0,T;W^{k+1,\infty}(\Omega;\R^d))\cap W^{1,1}(0,T;H^{k+1}(\Omega;\R^d)).
\]
Then,
\[
 \sup_{0\leq t \leq T} \RE{(\vuh|\tvu)}(t) \aleq  h^A, \quad \norm{\vuh-\tvu}_{L^\infty L^2} \aleq  h^{A/2},
\]
where $A$ is given in \eqref{rateA}. 
\end{thm}
\begin{proof}
By the projection estimates \eqref{piz} we have
\begin{equation}\label{initialRE}
 \RE(\vuh|\tvu)(0)
 =\frac12\norm{\PiZ \tvu(0)-\tvu(0)}_{L^2}^2 
 \aleq h^{2k+2} |\tvu(0)|_{H^{k+1}}^2.
\end{equation}
Recalling the consistency errors stated in Lemma \ref{lm_cons} and the abstract error estimates stated in  Theorem \ref{thm_re} completes the proof.  
\end{proof}

Note that for $RT_k/P_k$ finite element methods the optimal convergence rate is $k+1$ for the Stokes
equation~\cite{Wang} and $k+1-d/6$ for the stationary Navier--Stokes equations~\cite{Cai},  $k \geq 0.$ We also recall the result of Guzm\'an et al.~\cite{GSS}, which establishes the convergence rate $h^k$, $k\geq1$, for the incompressible Euler system. Further,  for a linearised inviscid problem, Barrenechea et al.~\cite{BBG} proved an $L^2$-error estimate of order $h^{k+1/2}$. 
Therefore, the convergence rate obtained in Theorem~\ref{thm_r1} is not sharp. 

\subsection{Convergence rate -- II}\label{sec_rate2}
We now derive an improved error estimate by a more careful use of the relative-energy identity. 
\begin{thm}[Convergence rate -- II]\label{thm_r2}
 Let $\vuh$ be a numerical solution of the $RT_k/P_k$ method \eqref{skm}. 
Let the Euler system admit a strong solution $\tvu$ in the sense of Definition~\ref{defi_Cs} with $m\geq k+3$, which satisfies the regularity class
\[ \tvu\in H^1(0,T;H^{k+1}(\Omega;\R^d))
 \cap L^\infty(0,T;W^{1,\infty}(\Omega;\R^d))
 \cap L^2(0,T;W^{k+1,\infty}(\Omega;\R^d)).\]
Then, the following error estimates hold 
\begin{equation}\label{rate2}
 \sup_{0\leq t \leq T} \RE{(\vuh|\tvu)}(t) \aleq  h^{2k+1}, \quad \norm{\vuh-\tvu}_{L^\infty L^2} \aleq  h^{k+1/2}.
\end{equation}
\end{thm}
\begin{remark}
This rate improves the convergence rate obtained in Theorem~\ref{thm_r1}. Moreover, it extends the result of Guzm\'an et al.~\cite{GSS} to the lowest-order case $k=0$, for which the rate $h^{1/2}$ is obtained, and improves their estimate by half an order for $k\geq1$.
 \end{remark}
\begin{proof}
We begin the proof with the decomposition of the error 
\[\eu=\vuh-\tvu=\deltau+\etau\quad\mbox{with}\quad
\deltau=\vuh-\Pitvu,\qquad \etau=\Pitvu-\tvu,\qquad \Pitvu=\PiR\tvu.
\]
As $\tvu$ is the strong solution, it holds for any $\vvarphih \in Z_h$ that 
\begin{equation*}
\inner{\pdt \tvu, \vvarphih }
 + \inner{ \tvu \cdot \Grad \tvu, \vvarphih }
= -\inner{\Grad \tp, \vvarphih}  
=0 ,
\end{equation*}
where we have used equality~\eqref{weakdiv}.  
Subtracting the above equality from \eqref{skm1} and setting the test function $\vvarphih = \deltau \in Z_h$ yields
\begin{equation}\label{re-e1}
\inner{ \pdt \eu ,  \deltau }  = \inner{\tvu \cdot \Grad \tvu, \deltau} - b(\vuh,\vuh;\deltau).
\end{equation}
Recalling that $\vuh = \etau+\deltau +\tvu$, identity~\eqref{udu}, and the fact that $\jumpB{\tvu}=0$, we have
\begin{align*}
 b(\vuh,\vuh;\deltau) = b(\vuh,\etau+\deltau+\tvu;\deltau) 
 =\seminorm{\deltau}_{\vuh}^2
 +b(\vuh,\etau;\deltau)
 + \inner{\vuh\cdot\Grad\tvu,\deltau}.
\end{align*}
Substituting the above equality into \eqref{re-e1}, we obtain
\begin{align*}
& \pdt  \RE(\vuh|\Pitvu)  + \seminorm{\deltau}_{\vuh}^2 =
\underbrace{-\inner{ \pdt \etau ,  \deltau } }_{=: R_1}  
- b(\vuh,\etau;\deltau)  
\underbrace{- \inner{\eu \cdot \Grad \tvu, \deltau}}_{=:R_2}
\\&= R_1+ R_2 + \underbrace{\inner{\vuh\cdot\Gradh\deltau, \etau} }_{=:R_3}  
\underbrace{- \innerfint{\vuh\cdot\vn \avs{\etau}, \jumpB{\deltau }   } }_{=:R_4} 
\underbrace{ - \innerfint{\abs{\vuh\cdot\vn} \jumpB{\etau}, \jumpB{\deltau }   } }_{=:R_5}.
\end{align*}
Consequently,
\begin{equation}\label{relative-discrete}
 \RE(\vuh|\Pitvu)(\tau) + \int_0^\tau  \seminorm{\deltau}_{\vuh}^2 \dt
 = \RE(\vuh|\Pitvu)(0) + \int_0^\tau \sum_{i=1}^5 R_i(t)\,\dt. 
\end{equation}
The first term on the right hand side is the initial interpolation error
\[
 \RE(\vuh|\Pitvu)(0)
 \aleq h^{2k+2} \norm{\tvu(0)}_{H^{k+1}}^2, 
\]
see \eqref{initialRE}. 
By H\"older's inequality, the interpolation estimate~\eqref{proj}, and Young's inequality, we have
\begin{align*}
\left|\int_0^\tau R_1\,\dt\right|
\leq h^{k+1}\norm{\pdt\tvu}_{L^2(0,\tau;H^{k+1})}  \norm{\deltau}_{L^2L^2} 
\leq   \int_0^\tau \RE(\vuh|\Pitvu)(t) \dt \; 
+ c h^{2k+2},
\end{align*}
where $c$ depends on $\norm{\tvu}_{H^1(0,T;H^{k+1})}$.  
Moreover, recalling $\eu=\etau+\deltau$, we obtain
\begin{align*}
\left|\int_0^\tau R_2\,\dt\right|
&\leq\int_0^\tau\norm{\Grad\tvu}_{L^\infty}
 \left(\norm{\etau}_{L^2}\norm{\deltau}_{L^2}
 +\norm{\deltau}_{L^2}^2\right)\,\dt\\
&\leq c_1 \int_0^\tau \RE(\vuh|\Pitvu)(t) \dt \; + c_1 c_2 h^{2k+2},
\end{align*}
where $c_1$ depends on $\norm{\tvu}_{L^\infty W^{1,\infty}}$ and $c_2$ depends on $\norm{\tvu}_{L^2H^{k+1}}$.

We next estimate the face terms. Recalling the H\"older's inequality, the trace inequality, and the interpolation estimate~\eqref{proj}, we have
\begin{align*}
\left|\int_0^\tau \!\! (R_4+R_5)\dt\right|
&\leq \left( \int_0^\tau \!\! \seminorm{\deltau}_{\vuh}^2\dt\!\! \right)^{1/2} \!\!\!
\left(\int_0^\tau\!\! \innerfint{|\vuh\cdot\vn|,
 |\jumpB{\etau}|^2+|\avs{\etau}|^2}\,\dt\right)^{1/2}
\\& \leq \frac14\int_0^\tau \seminorm{\deltau}_{\vuh}^2\,\dt + h^{-1}\norm{\vuh}_{L^\infty L^2} (h^{k+1}\norm{\tvu}_{L^2 W^{k+1,4}})^2
\\&\leq \frac14\int_0^\tau \seminorm{\deltau}_{\vuh}^2\,\dt
+c h^{2k+1},
\end{align*}
where $c$ depends on $\norm{\tvu}_{L^2 W^{k+1,4}}$.  

We are left with the estimate of $R_3$. For $k=0$, we know from Lemma~\ref{lem_Zh} that $\deltau|_K\in P_0^d(K)$ is a piecewise constant vector field. Thus, $\Gradh\deltau=0$ and, consequently, $R_3=0$. 
Let now $k\geq1$, and define the elementwise mean
\[
 \overline{\tvu}_K:=\frac1{|K|}\int_K\tvu\,\dx.
\]
As $\deltau \in Z_h$, recalling Lemma~\ref{lem_Zh}, we have $\Gradh\deltau|_K\in P_{k-1}^{d\times d}(K)$. Hence, the interior moment condition \eqref{rt_dofs_cell} yields
\[
 \sum_{K\in\grid}
 \int_K\big((\overline{\tvu}_K\cdot\Grad)\deltau\big)\cdot\etau\,\dx=0.
\]
Using this cancellation, the identity $\vuh=\deltau+\Pitvu$, and the inverse estimate, we obtain for a.a. $t\in(0,T)$
\begin{align*}
|R_3|
&\leq\sum_{K\in\grid}\norm{\Gradh\deltau}_{L^2(K)}
\Big(\norm{\deltau}_{L^2(K)}\norm{\etau}_{L^\infty(K)}
+\norm{\Pitvu-\overline{\tvu}_K}_{L^\infty(K)}
 \norm{\etau}_{L^2(K)}\Big)\\
&\aleq \norm{\tvu}_{W^{1,\infty}}
 \norm{\deltau}_{L^2}^2
 +h^{k+1}\norm{\tvu}_{W^{1,\infty}}
 \norm{\tvu}_{H^{k+1}}\norm{\deltau}_{L^2},
\end{align*}
where we have used
\[
 \norm{\Pitvu-\overline{\tvu}_K}_{L^\infty(K)}
 \leq\norm{\Pitvu-\tvu}_{L^\infty(K)}
 +\norm{\tvu-\overline{\tvu}_K}_{L^\infty(K)}
 \aleq h\norm{\tvu}_{W^{1,\infty}(K)}.
\]
After integration in time and using Young's inequality, it follows that
\begin{align*}
\left|\int_0^\tau R_3\,\dt\right|
&\leq c_1  \int_0^\tau \RE(\vuh|\Pitvu)(t) \dt \; 
 +c_1 c_2 h^{2k+2},
\end{align*}
where $c_1$ depends $\norm{\tvu}_{L^\infty W^{1,\infty}}$ and $c_2$ depends on $\norm{\tvu}_{L^2 H^{k+1}}$.

Collecting the above estimates into~\eqref{relative-discrete}, we obtain, for every $\tau\in[0,T]$,
\begin{align*}
\RE(\vuh|\Pitvu)(\tau)
+\frac34\int_0^\tau\seminorm{\deltau}_{\vuh}^2\,\dt
&\leq c_1 \int_0^\tau \RE(\vuh|\Pitvu)(t)\,\dt +c_2 h^{2k+1},
\end{align*}
where $c_1$ depends $\norm{\tvu}_{L^\infty W^{1,\infty}}$ and $c_2$ depends on $\norm{\tvu}_{L^\infty W^{1,\infty}}$ and $\norm{\tvu}_{L^2 W^{k+1,\infty}}$. 
Now, Gronwall's lemma yields
\[
 \sup_{0\leq t\leq T}\RE(\vuh|\Pitvu)(t)
 +\int_0^T\seminorm{\deltau}_{\vuh}^2\,\dt
 \aleq h^{2k+1}.
\]
Finally, using the interpolation estimate~\eqref{proj} and the decomposition $\eu=\deltau+\etau$, we obtain
\[
 \sup_{0\leq t\leq T}\RE(\vuh|\tvu)(t)\aleq h^{2k+1},
 \quad\mbox{and} \quad
 \norm{\vuh-\tvu}_{L^\infty L^2}\aleq h^{k+1/2}.
\]
This completes the proof.
\end{proof}


\section{Numerical experiments}\label{con_num}
In this section, we present numerical experiments to illustrate our theoretical convergence results. The method is implemented with the finite element software package NGSolve, available at  \url{https://ngsolve.org/}. 

 \subsection*{Experiment 1} In the first experiment, we study the Taylor--Green flow in the domain $\Omega = [0,2\pi]^2$ with no-flux boundary conditions. The reference solution
\begin{equation*}
    \tvu =
\begin{pmatrix}
    \sin(x_1) \cos(x_2) \\
    -\cos(x_1)\sin(x_2)
\end{pmatrix}\exp(-2t/\lambda), \quad  p= \frac14(\cos(2x_1) +\cos(2x_2))\exp(-4t/\lambda)
\end{equation*}
is driven by the external force $\vf=-(2/\lambda)\exp(-2t/\lambda)\tvu(0)$, where the constant $\lambda$ is set to be $100$. Accordingly, the term $\inner{\vf,\vvarphih}$ is added to the right-hand side of \eqref{skm1}. 
Since the same forcing term is included in both the continuous and discrete formulations, it cancels in the error identity. 
In the consistency argument it produces only a higher-order projection residual and hence does not change the stated rates.

Figure~\ref{fig-TG} presents the streamlines and pressure contours for $t=0,1$ with mesh size $h=0.0926$. It is clear that the vortex structure is well maintained, and the magnitude is slightly dissipated over time, as expected. 
 Table \ref{tab-TG} illustrates the numerical errors and the corresponding convergence orders of the velocity and pressure in the $L^2$-norm at time $T=1$ with time step $\Delta t =1/160$. 
 The observed convergence rates agree with those reported by Guzm\'an et al.~\cite{GSS} for the $RT_k/P_k$ method. Compared with the theoretical result obtained in Theorem~\ref{thm_r2}, where the convergence rate $\order(h^{k+1/2})$ was proved rigorously, the experimental convergence rates are approximately half an order higher. 
\begin{figure}[htbp]
    \centering
    \includegraphics[width=0.45\linewidth]{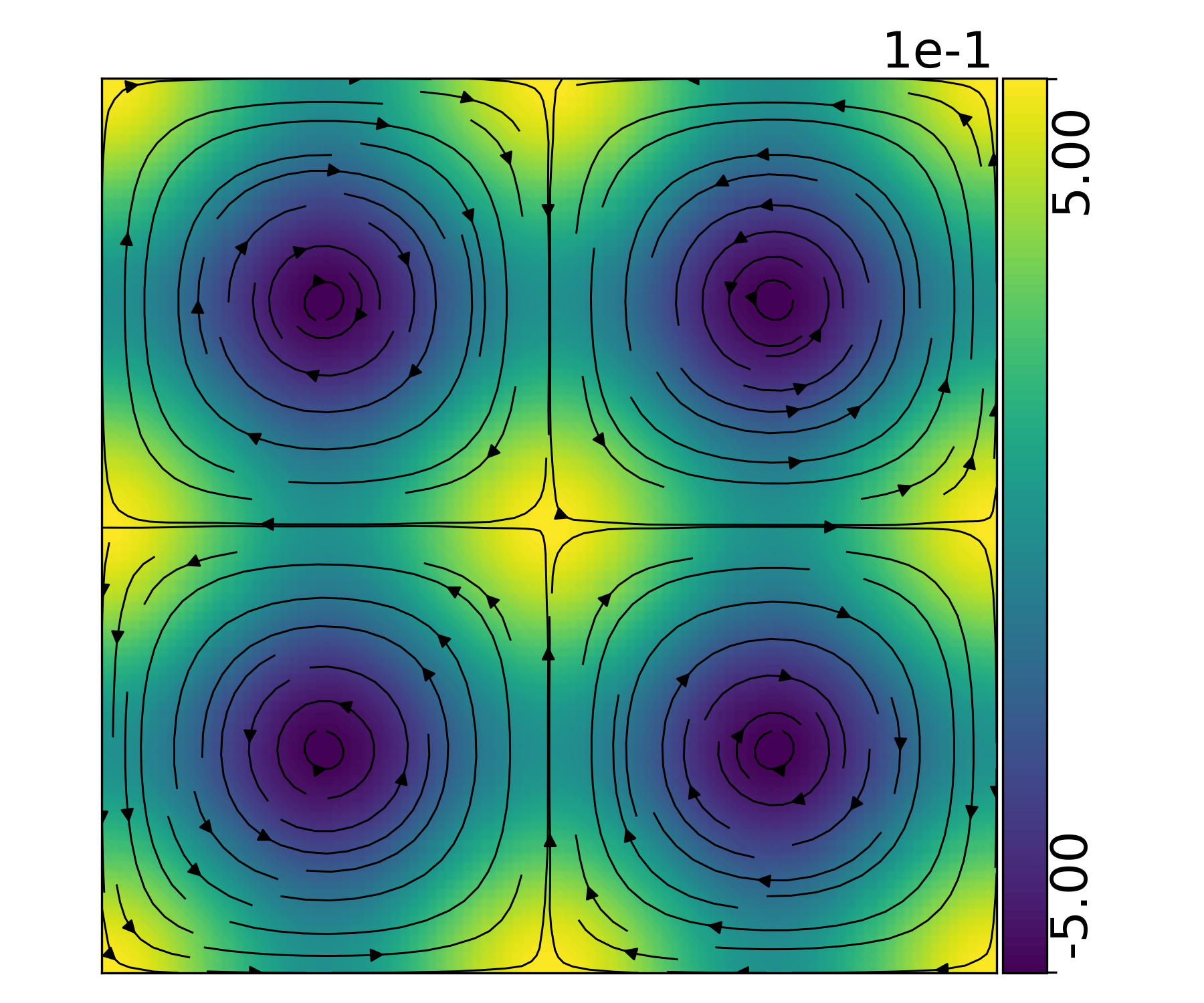}
    \includegraphics[width=0.45\linewidth]{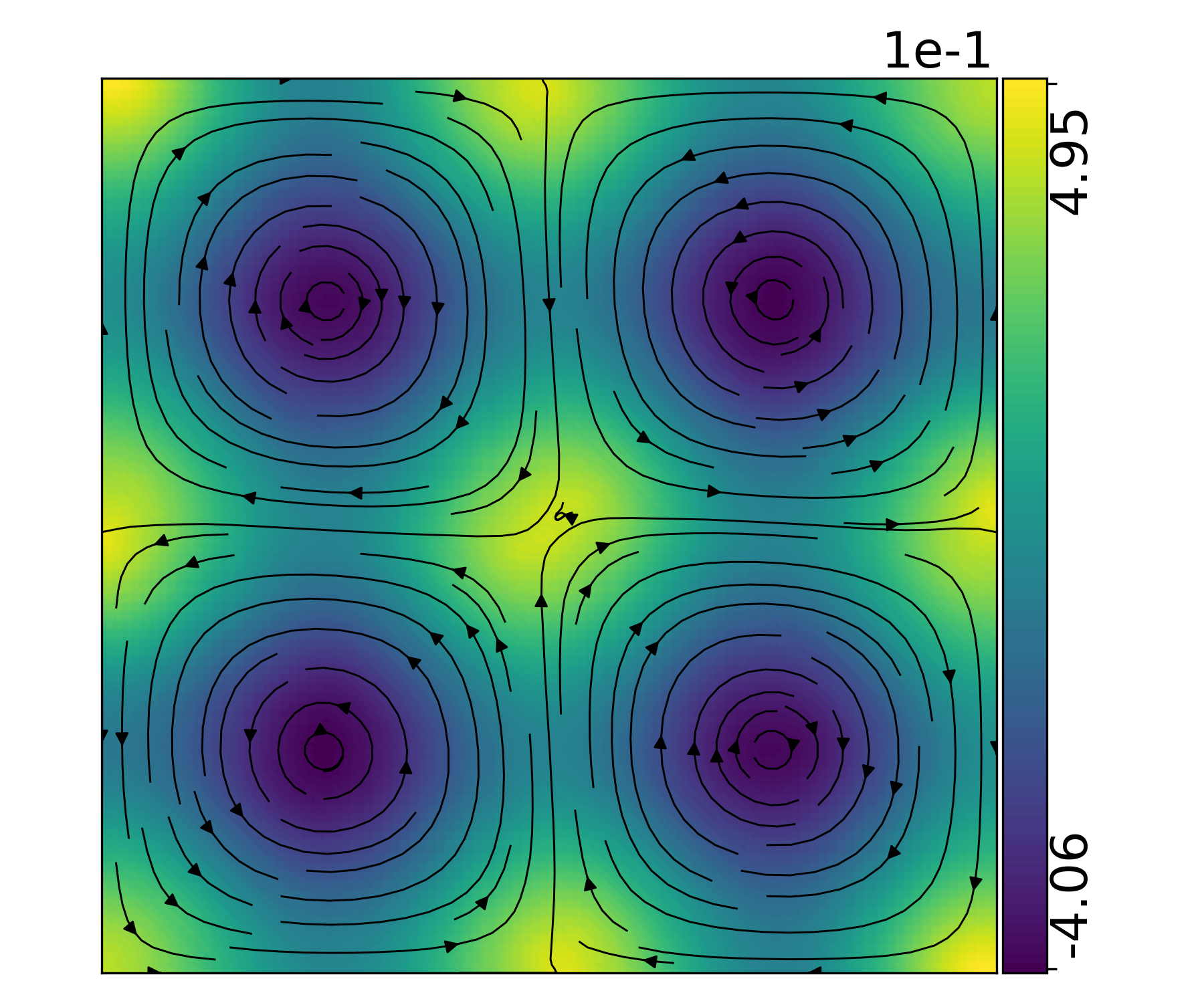}
    \caption{Experiment 1 (Taylor--Green vortex). Time evolution of the streamlines and pressure contours; from left to right, $t=0,1$.}
    \label{fig-TG}
\end{figure}
\begin{table}[htbp]
  \centering
  \caption{Experiment 1 (Taylor--Green vortex). Numerical error and corresponding convergence order with $RT_k/P_k$ elements with $k=0,1$,  $T = 1.0$,  $\Delta t = 6.25\times 10^{-3}$.}
  \label{tab-TG}
  \footnotesize
  \setlength{\tabcolsep}{3pt}
  \begin{tabular}{|c|cc|cc|cc|cc|}
    \hline
        \multirow{2}{*}{$h$} &
     \multicolumn{4}{c|}{$k = 0$} &
      \multicolumn{4}{c|}{$k = 1$} \\
   \cline{2-9}
    & $\|e_{\vu}\|_{L^2}$ & order & $\|e_p\|_{L^2}$ & order &
    $\|e_{\vu}\|_{L^2}$ & order & $\|e_p\|_{L^2}$ & order \\
    \hline
       0.7405 & 1.87e+00 &        & 1.08e+00 &        &
                 1.87e-01 &        & 1.11e-01 &        \\
     0.3702 & 1.11e+00 & 0.759  & 6.86e-01 & 0.653  &
              4.18e-02 & 2.16   & 2.83e-02 & 1.98    \\
     0.1851 & 6.07e-01 & 0.865  & 3.91e-01 & 0.812  &
              1.04e-02 & 2.01   & 7.23e-03 & 1.97    \\
     0.0926 & 3.23e-01 & 0.912  & 2.11e-01 & 0.886  &
              2.62e-03 & 1.99   & 1.97e-03 & 1.88    \\
    \hline
  \end{tabular}
\end{table}

 \subsection*{Experiment 2} In the second experiment, we study the shear layer problem taken from Guzm\'an et al.~\cite{GSS} with periodic boundary conditions and the following initial data
\begin{equation*}
    u_1 =
 \begin{cases}
    \tanh((x_2-\pi/2)/\vr), & x_2 \leq \pi, \\
    \tanh((3\pi/2-x_2)/\vr), & x_2 > \pi,
\end{cases} 
\quad  u_2= \delta \sin(x_1), \delta =0.05, \vr=\pi/15. 
\end{equation*}
In Figure~\ref{fig:exp2-sp}, we present the time evolution of streamlines and pressure contours for the $RT_k/P_k$ method with $k=0,1,2$, from which we observe clearer vortex patterns for $k=1,2$.  To further illustrate the distribution of vortices, we present in Figure~\ref{fig:exp2-vortex} the vorticity $\omega_h:=\operatorname{curl}(\vuh)=\partial_{x_1}u_{2,h}-\partial_{x_2}u_{1,h}$ at time $t=6$ for $k=1,2$. 

\begin{figure}
    \centering
    \includegraphics[width=0.325\linewidth]{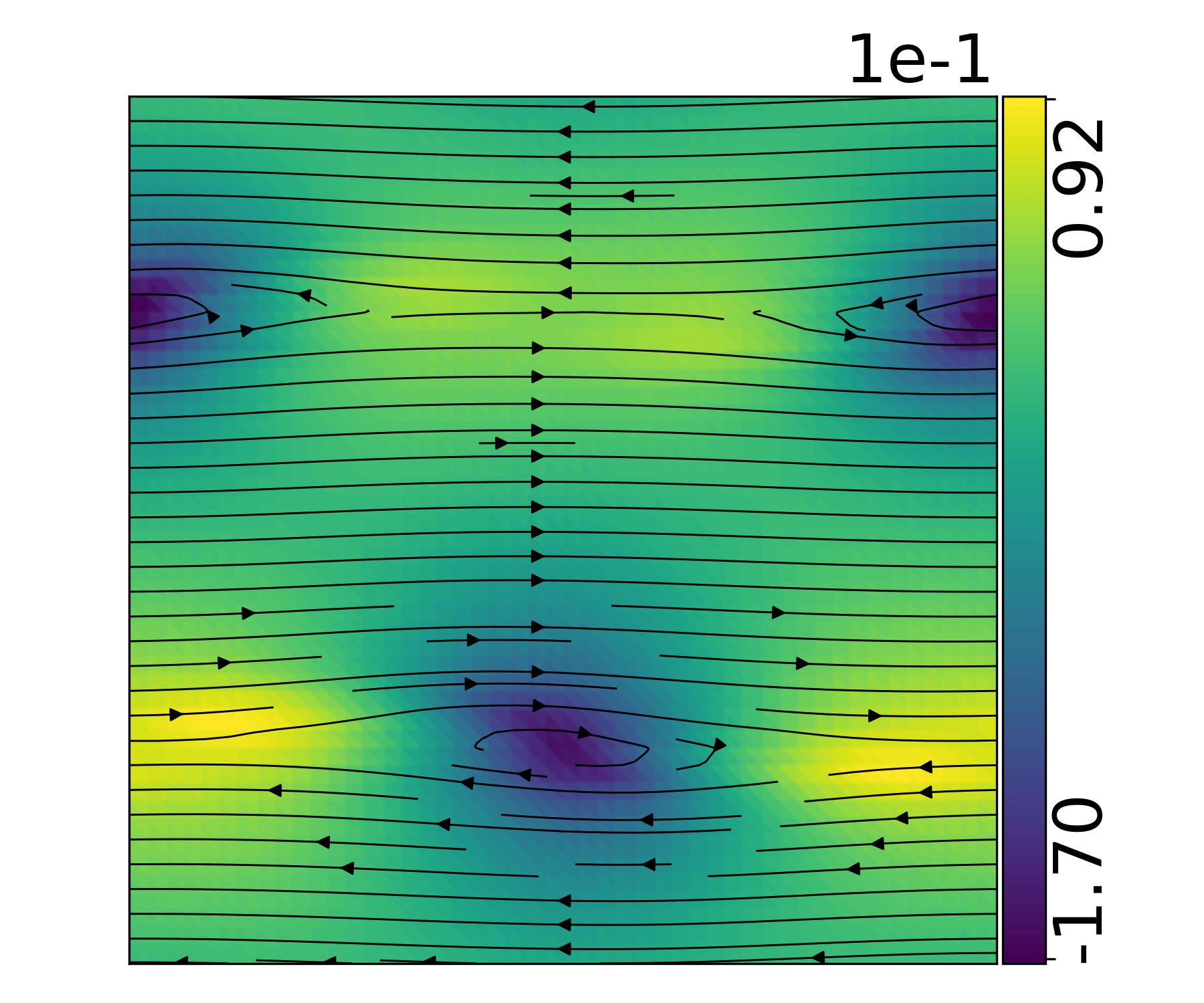}	
    \includegraphics[width=0.325\linewidth]{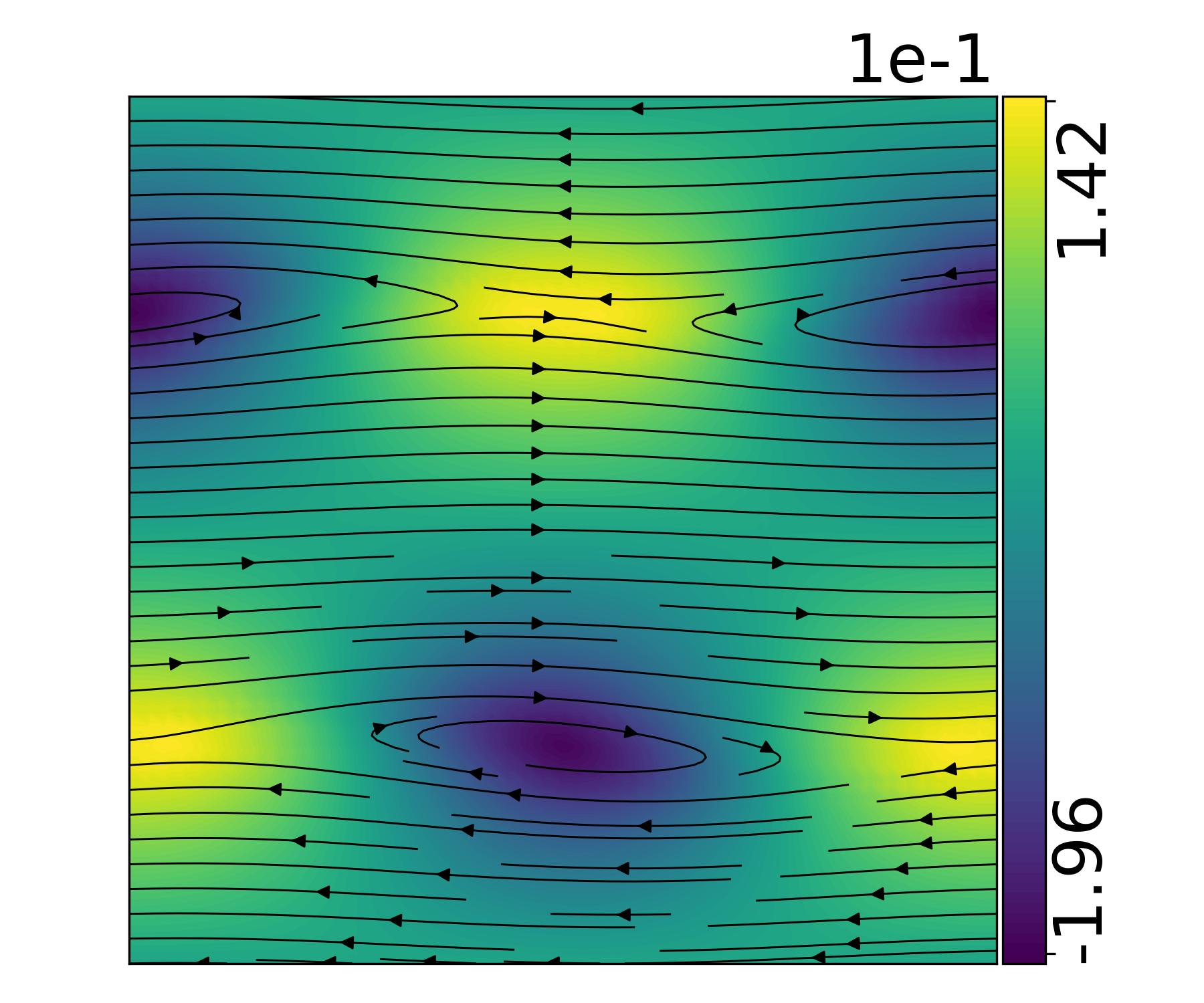}
    \includegraphics[width=0.325\linewidth]{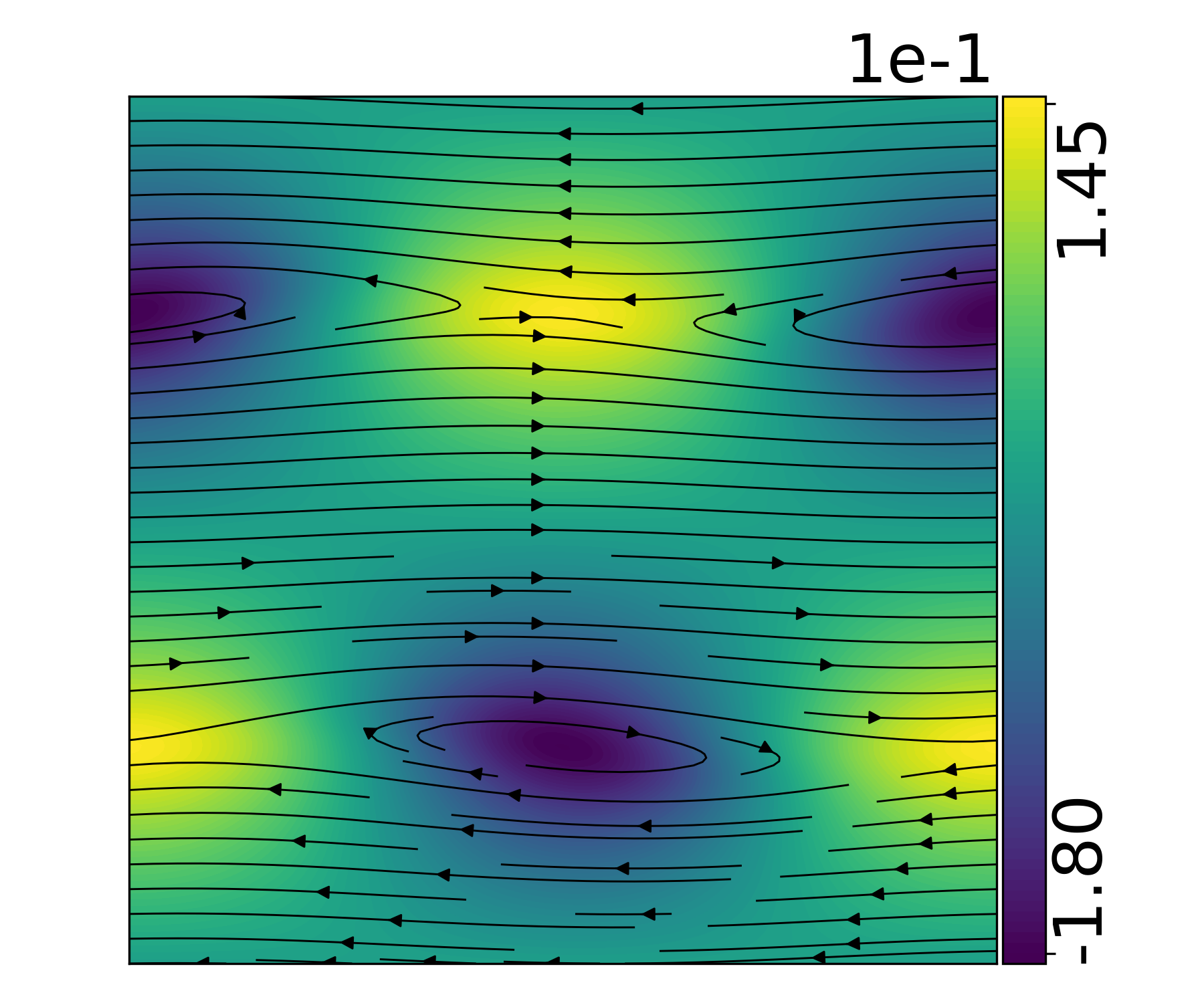}\\
     \includegraphics[width=0.325\linewidth]{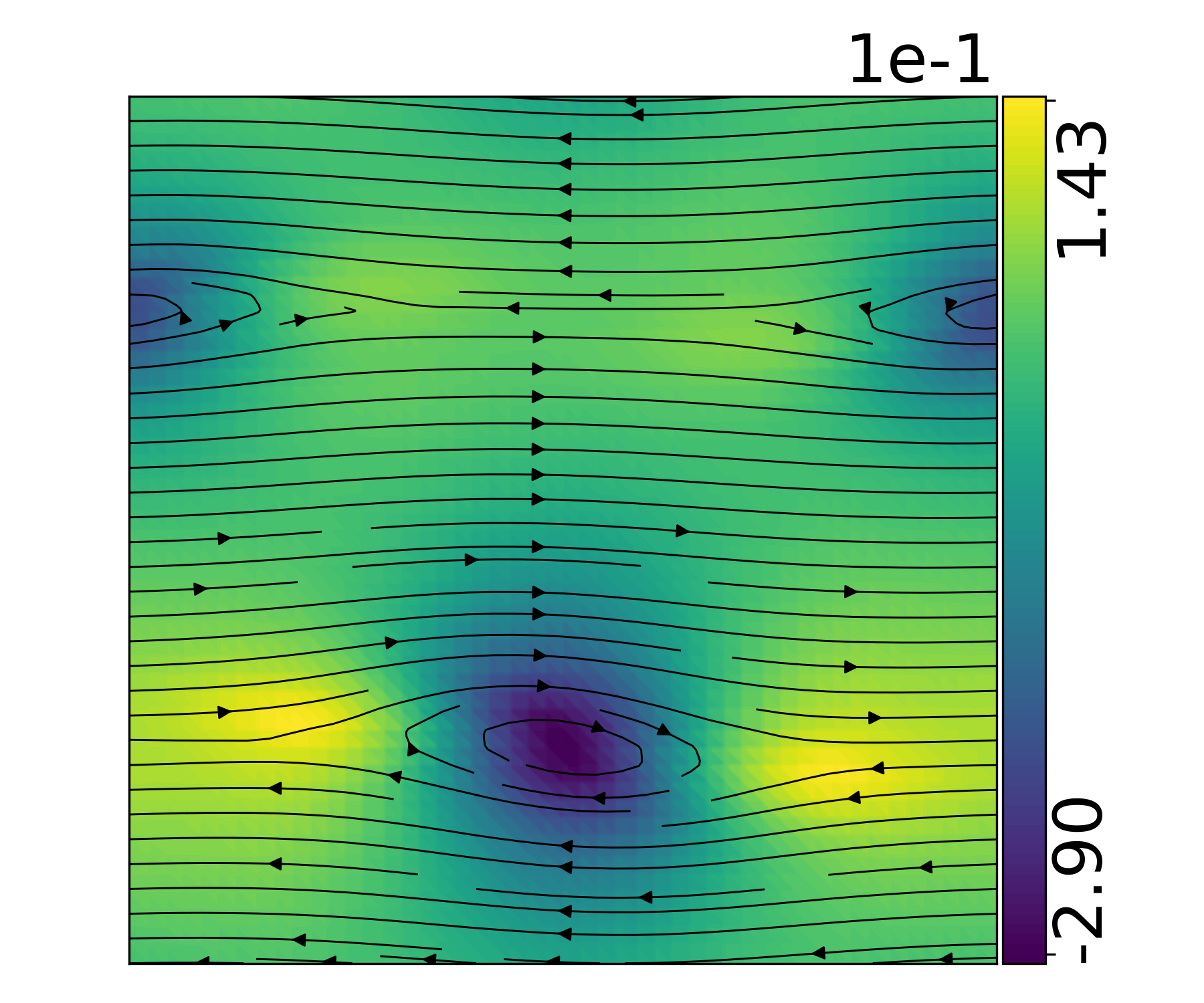}	
    \includegraphics[width=0.325\linewidth]{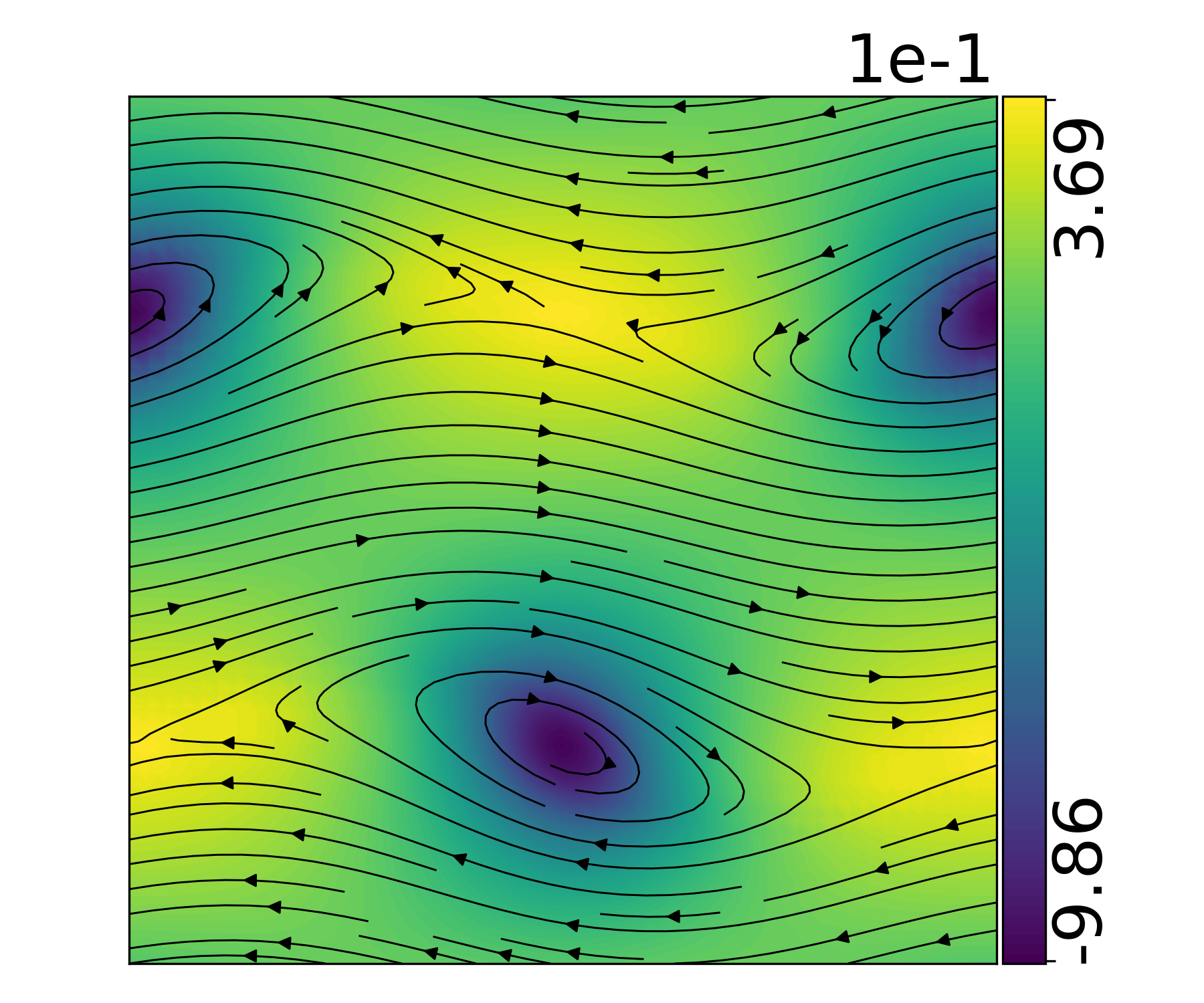}
    \includegraphics[width=0.325\linewidth]{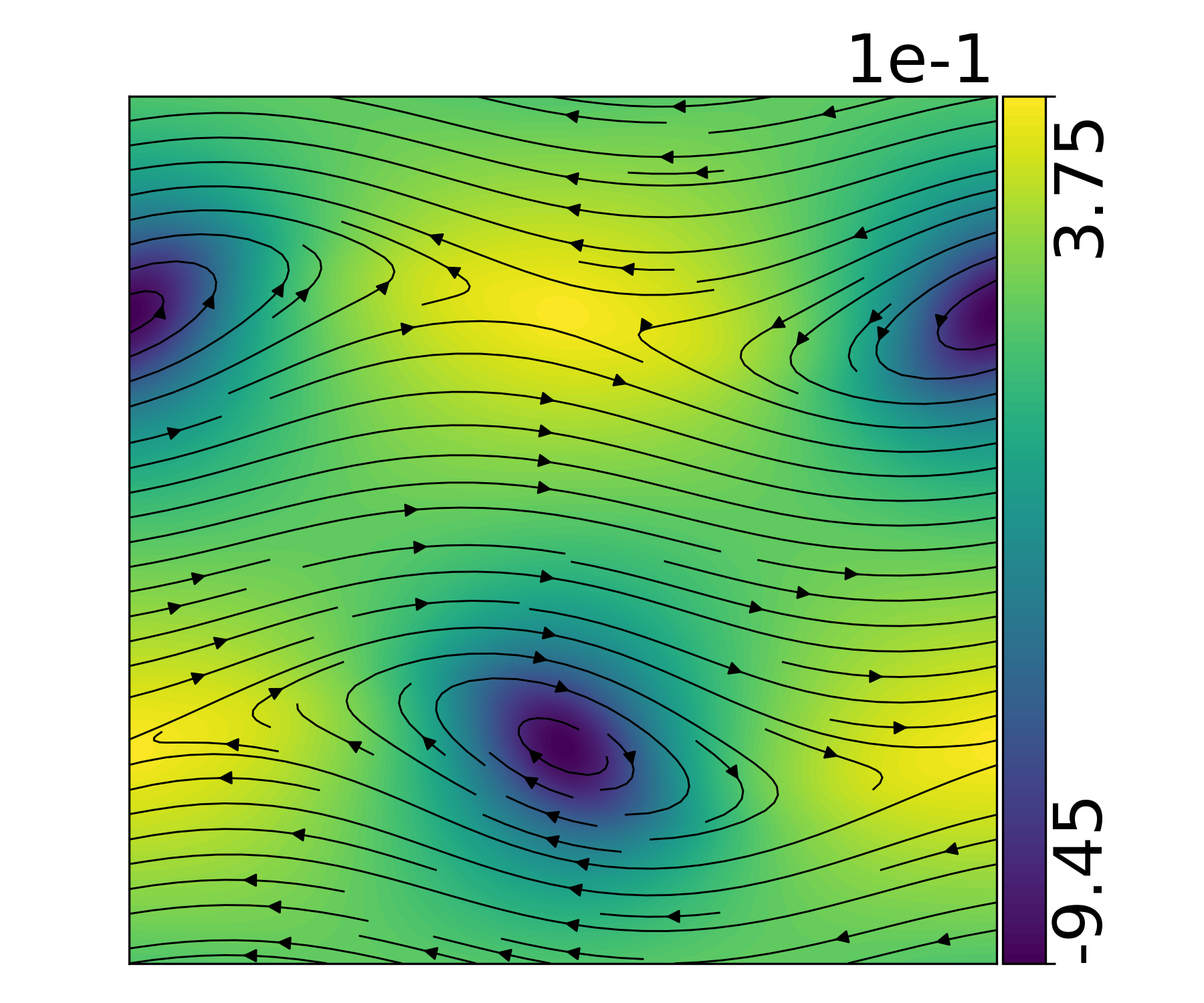}\\
     \includegraphics[width=0.325\linewidth]{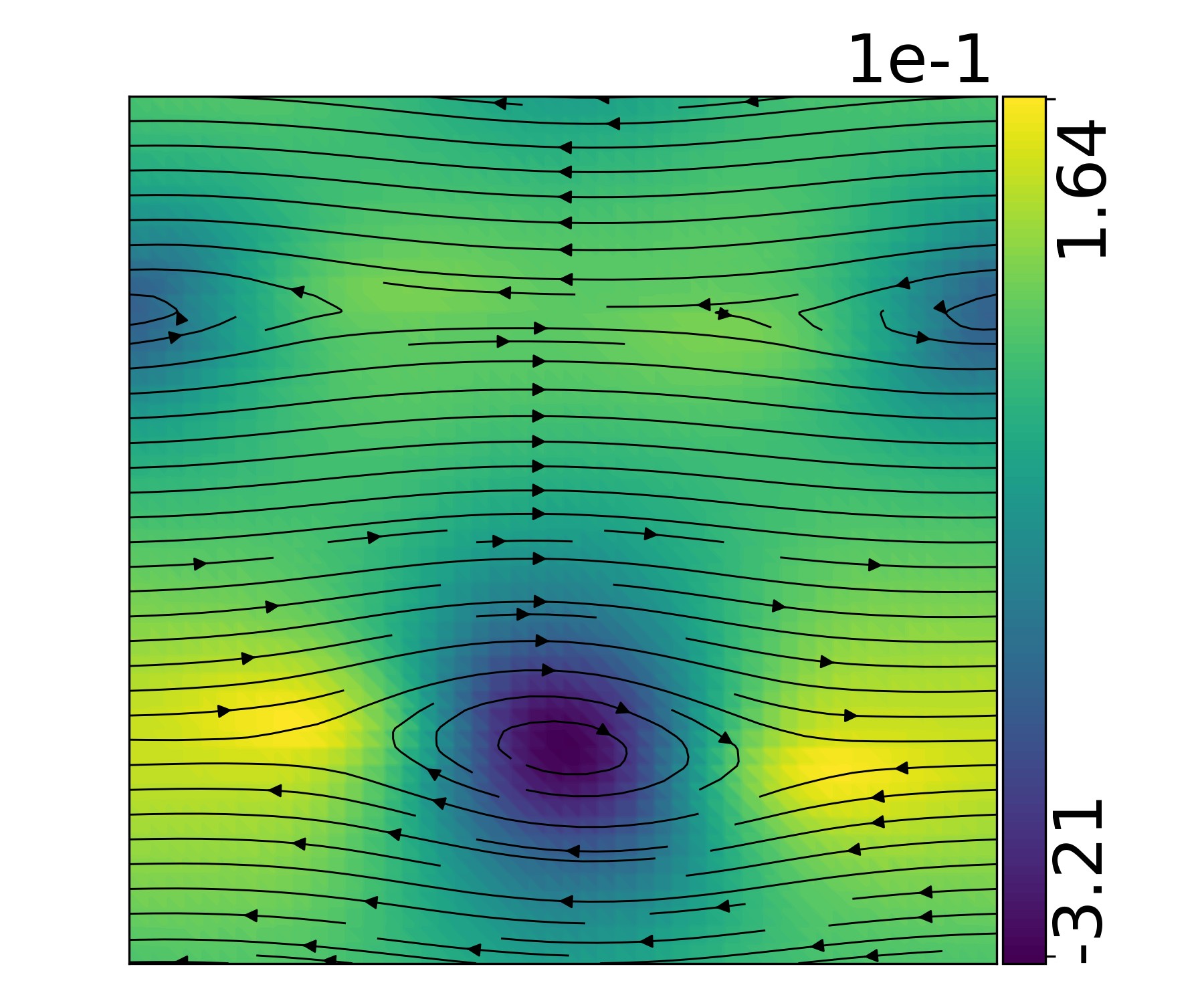}	
    \includegraphics[width=0.325\linewidth]{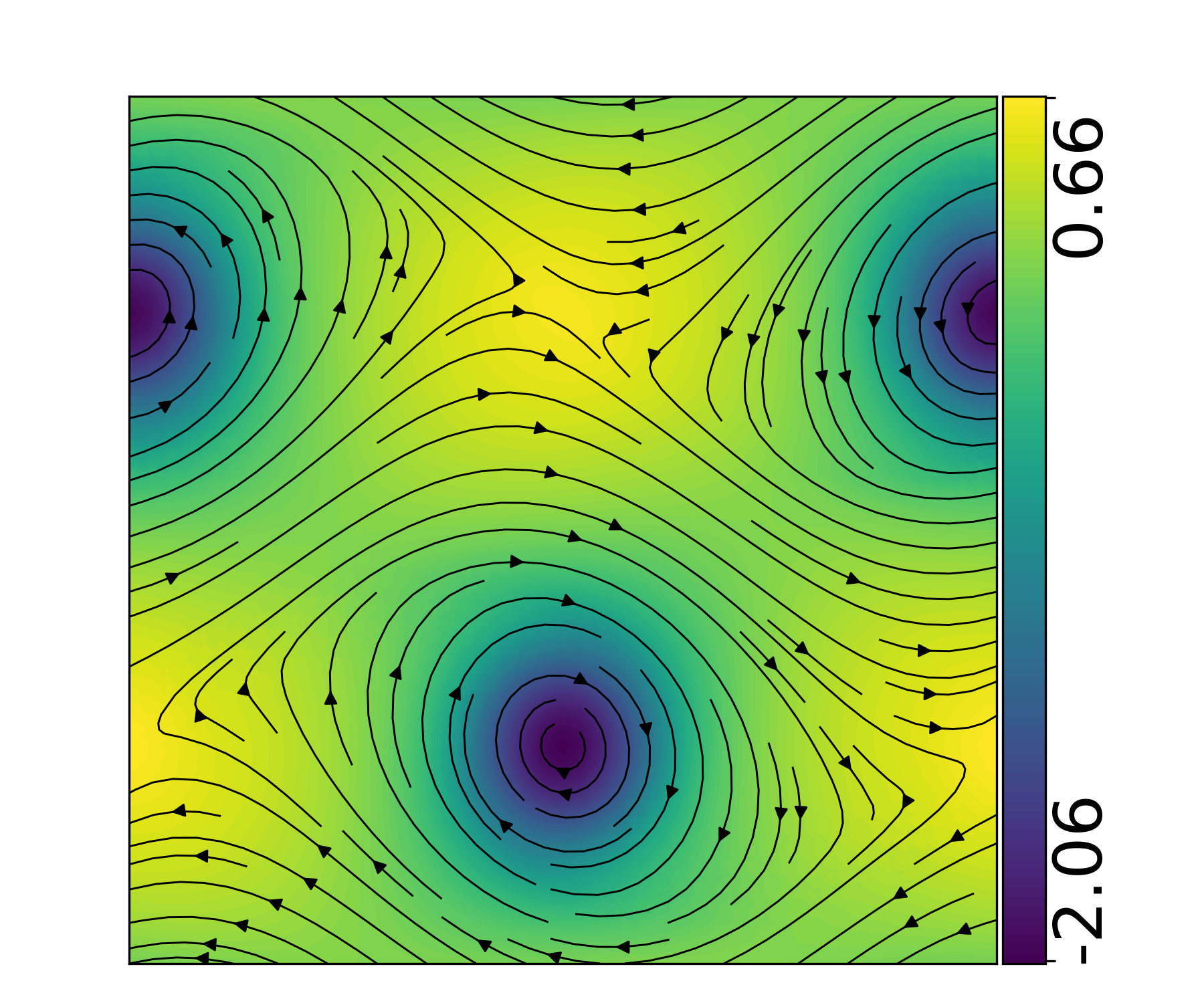}
    \includegraphics[width=0.325\linewidth]{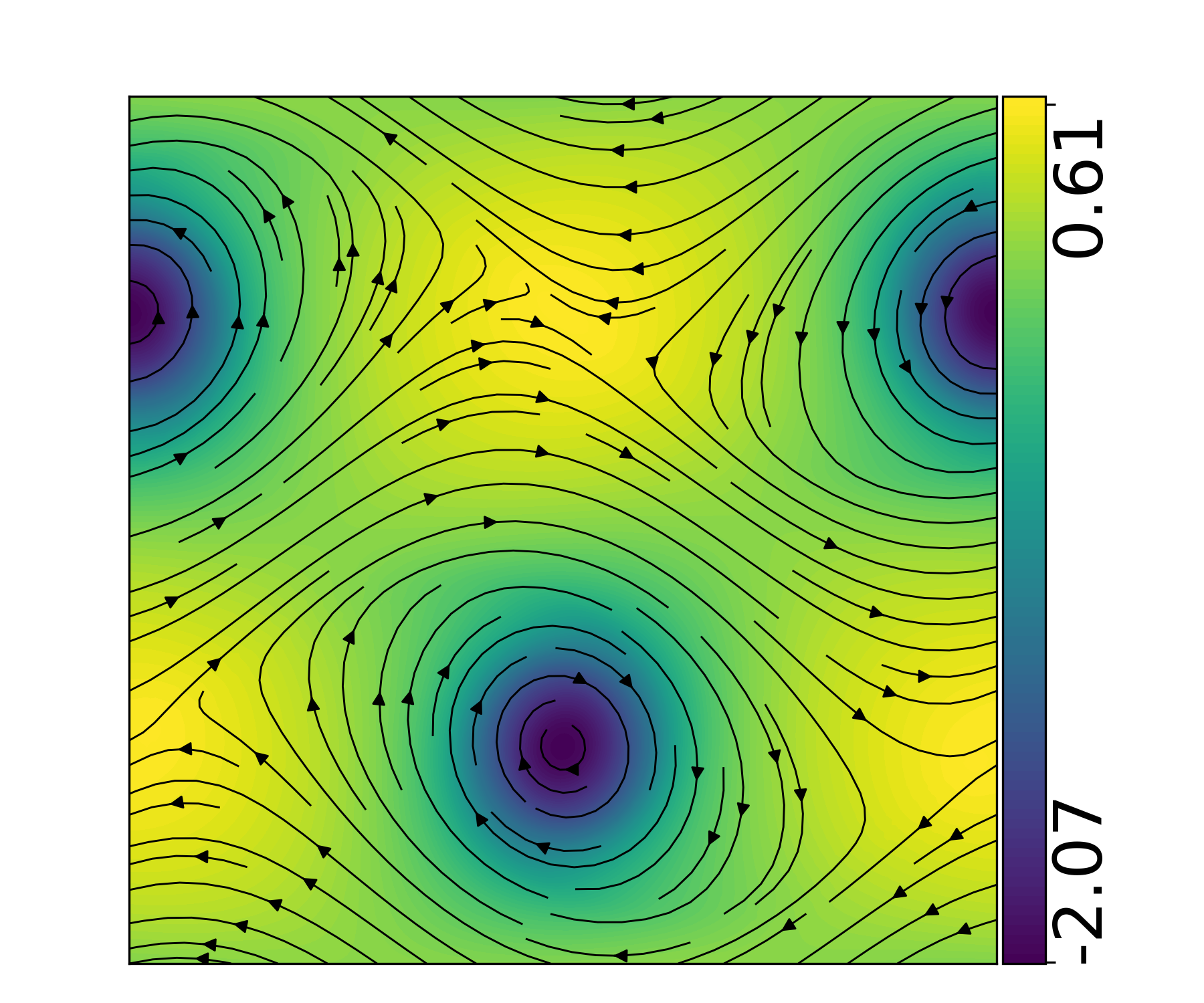}\\
     \includegraphics[width=0.325\linewidth]{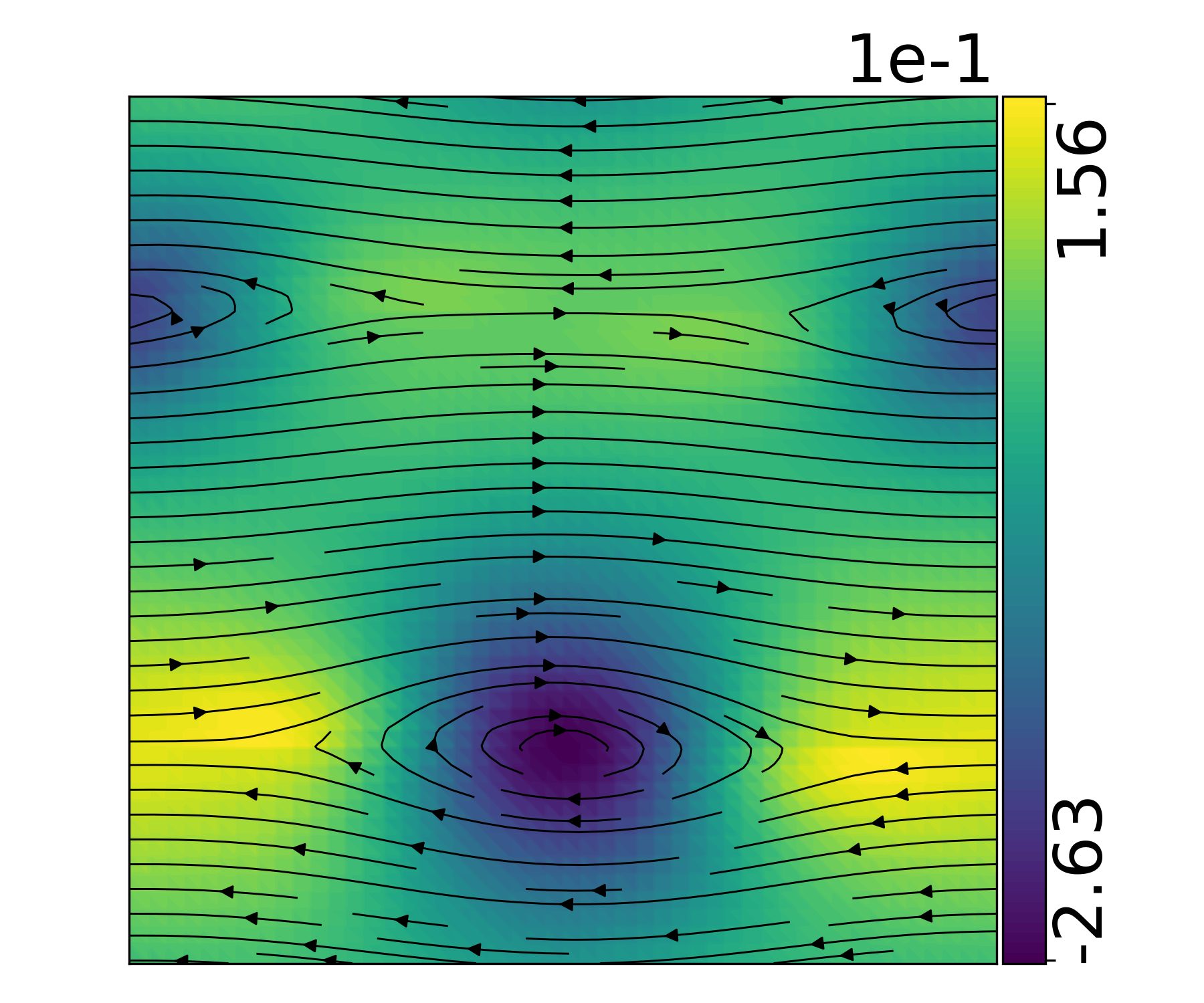}	
    \includegraphics[width=0.325\linewidth]{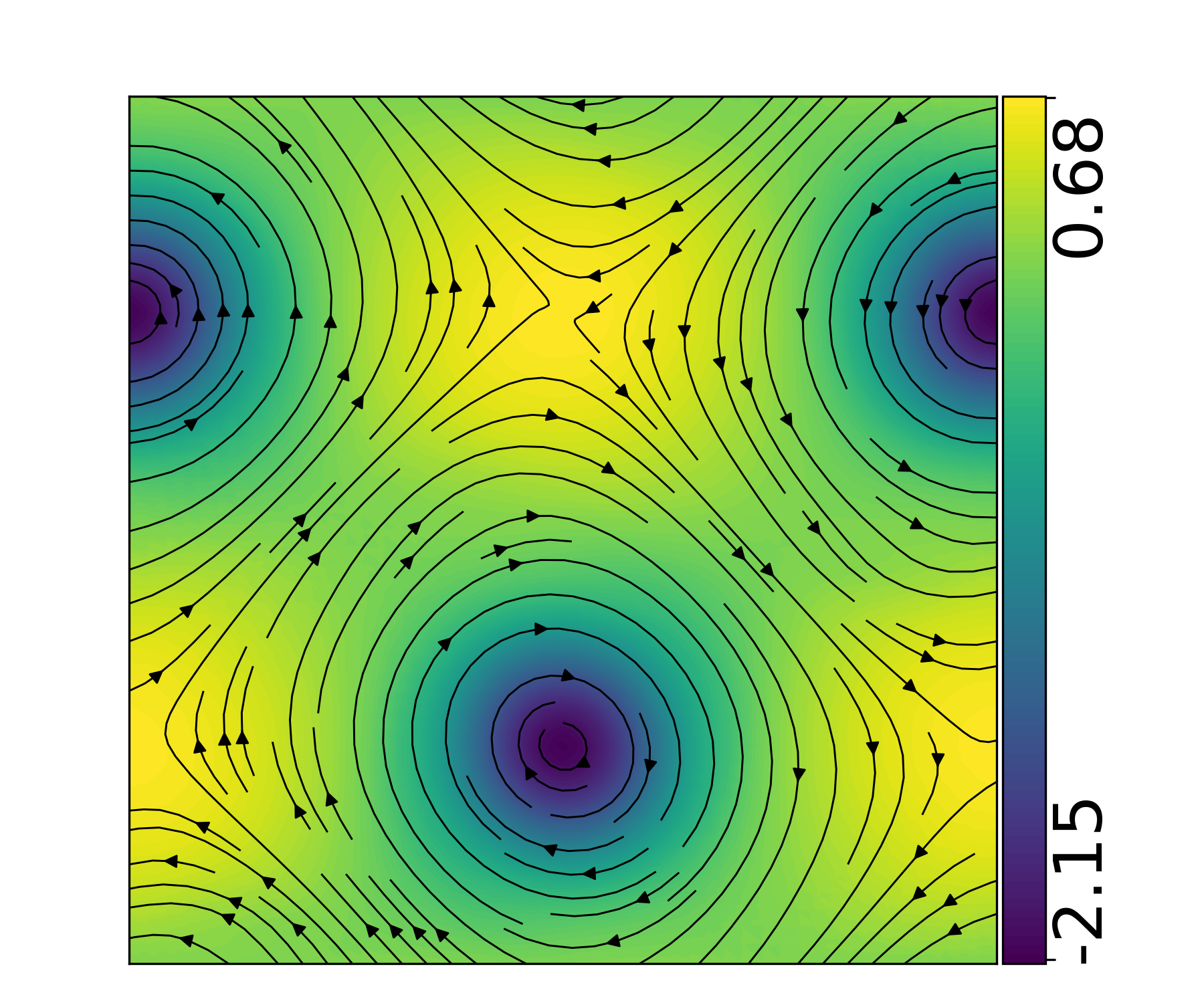}
    \includegraphics[width=0.325\linewidth]{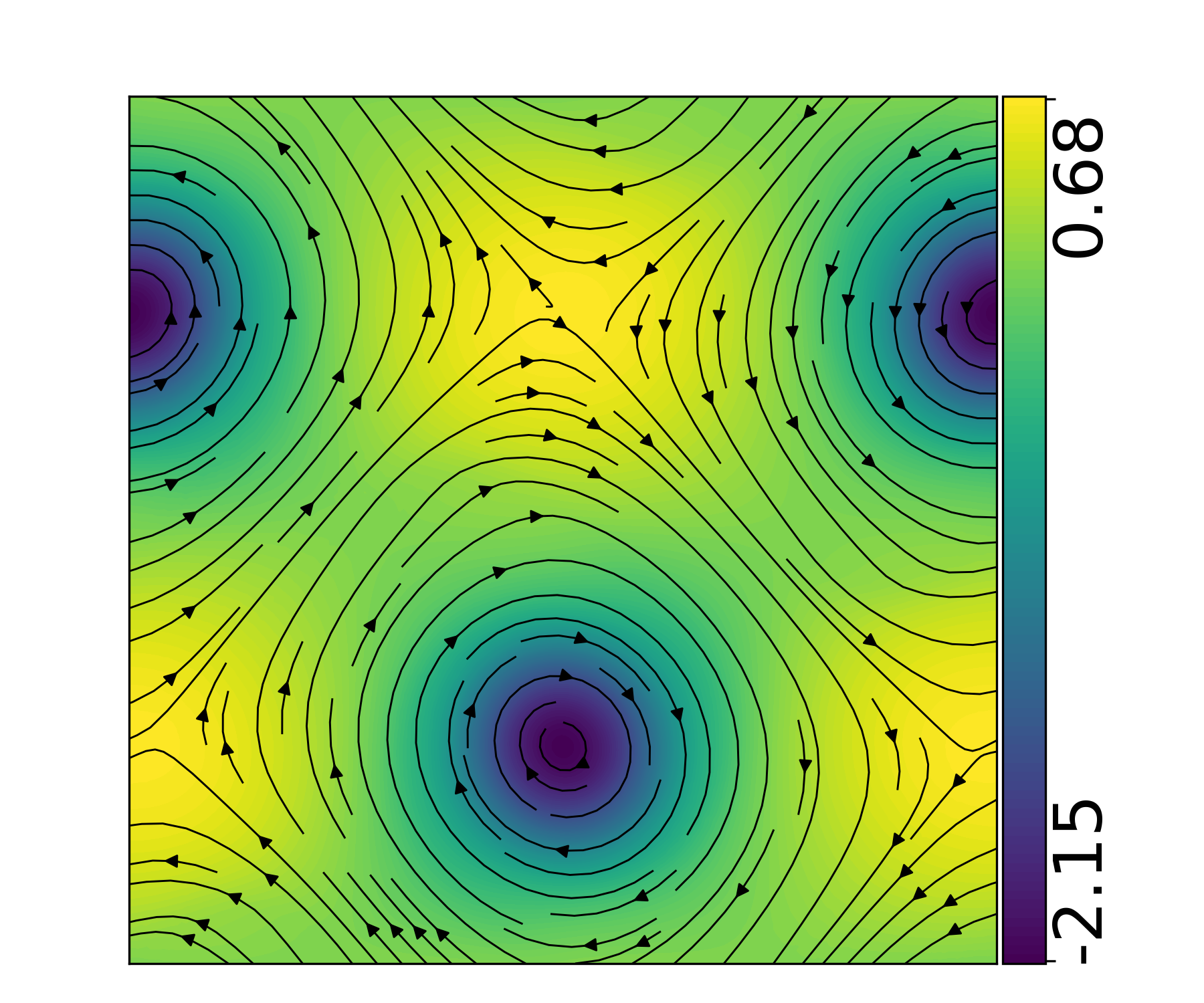}
    \caption{Experiment 2. Streamlines and pressure contours computed by the $RT_k/P_k$ method with $h=0.1851$.  From top to bottom, $t=2,4,6,8$; from left to right, $k=0,1,2$.}
    \label{fig:exp2-sp}
\end{figure}
\begin{figure}
    \centering
    \includegraphics[width=0.425\linewidth]{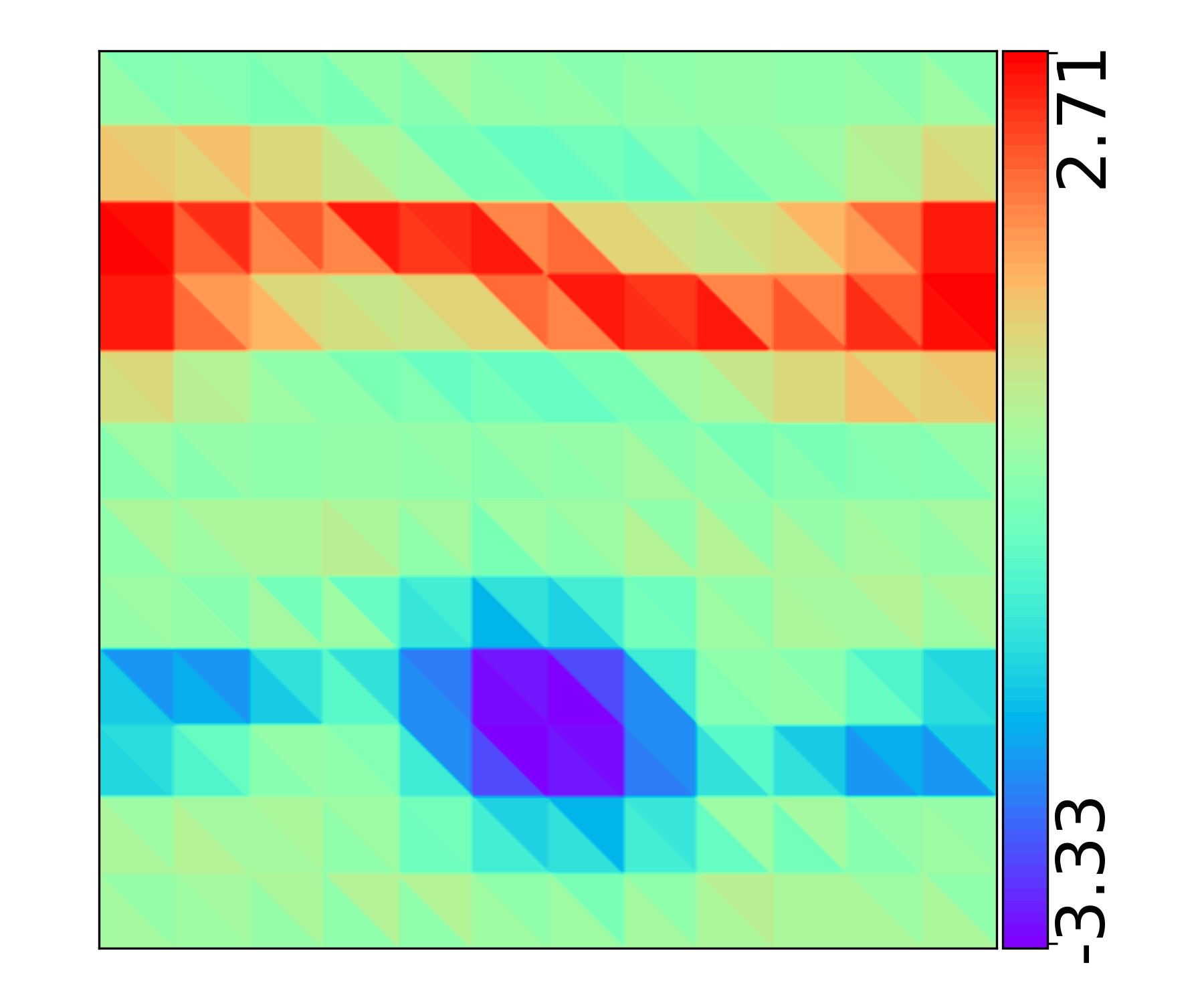}
    \includegraphics[width=0.425\linewidth]{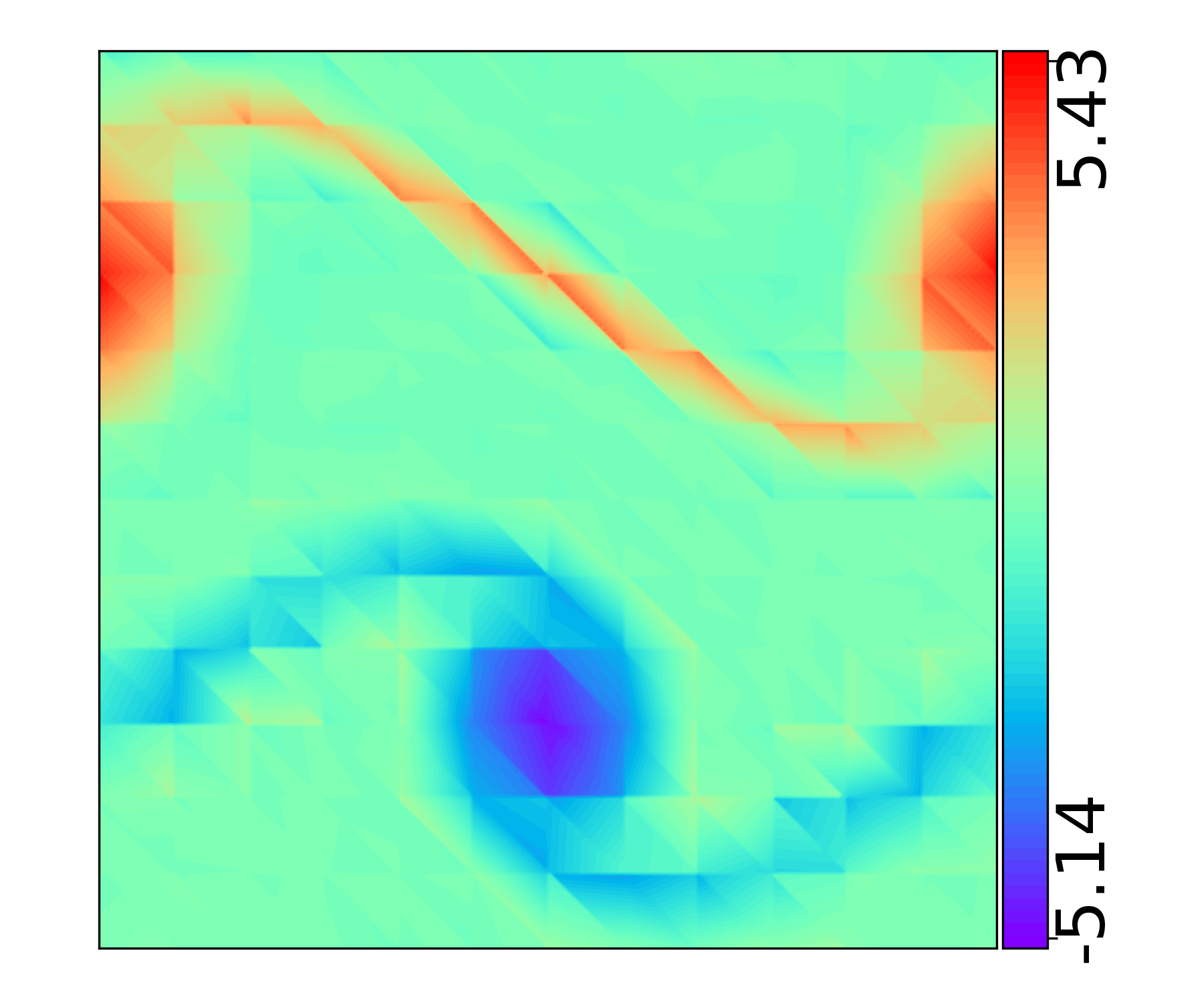}\\
    \includegraphics[width=0.425\linewidth]{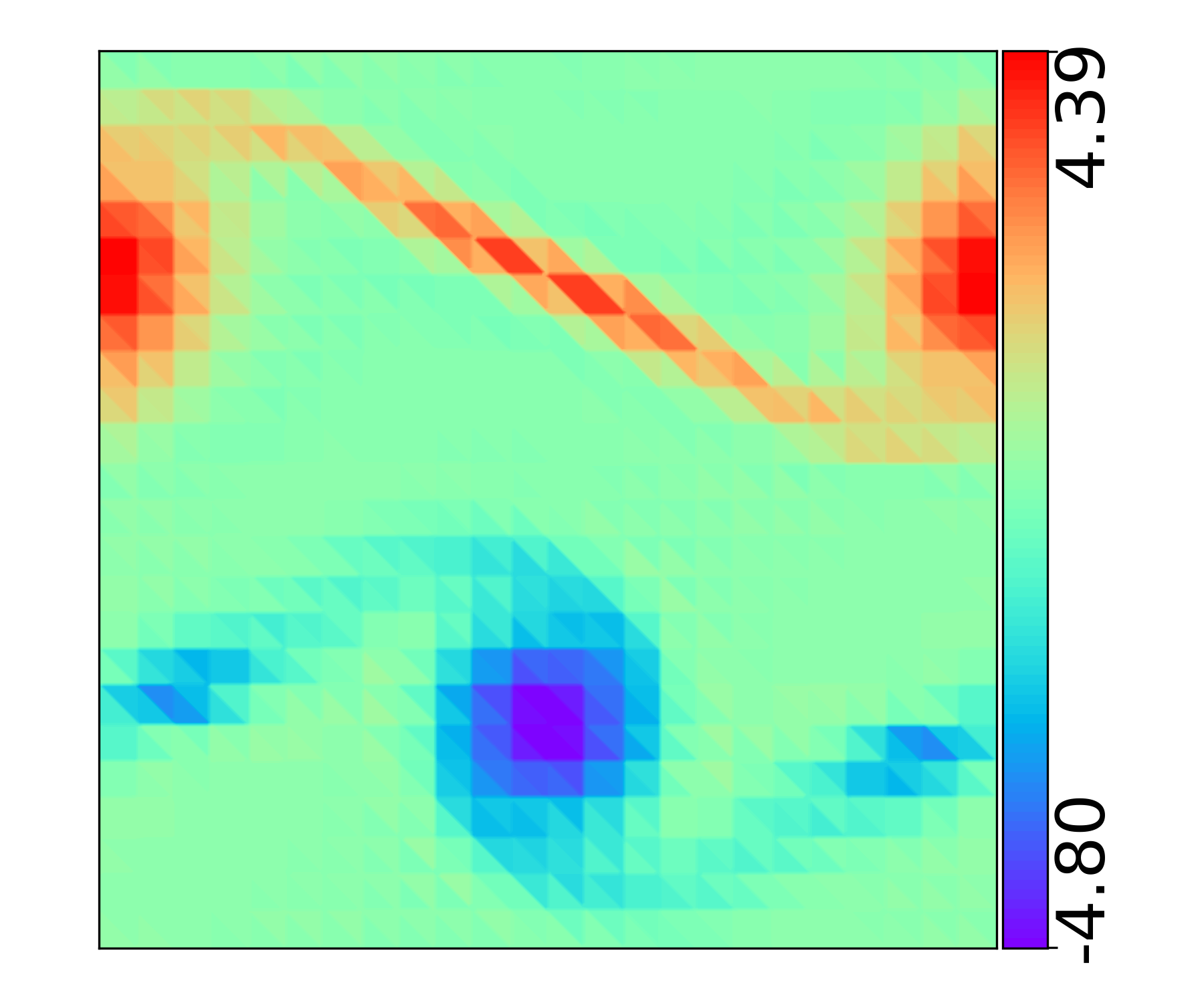}
    \includegraphics[width=0.425\linewidth]{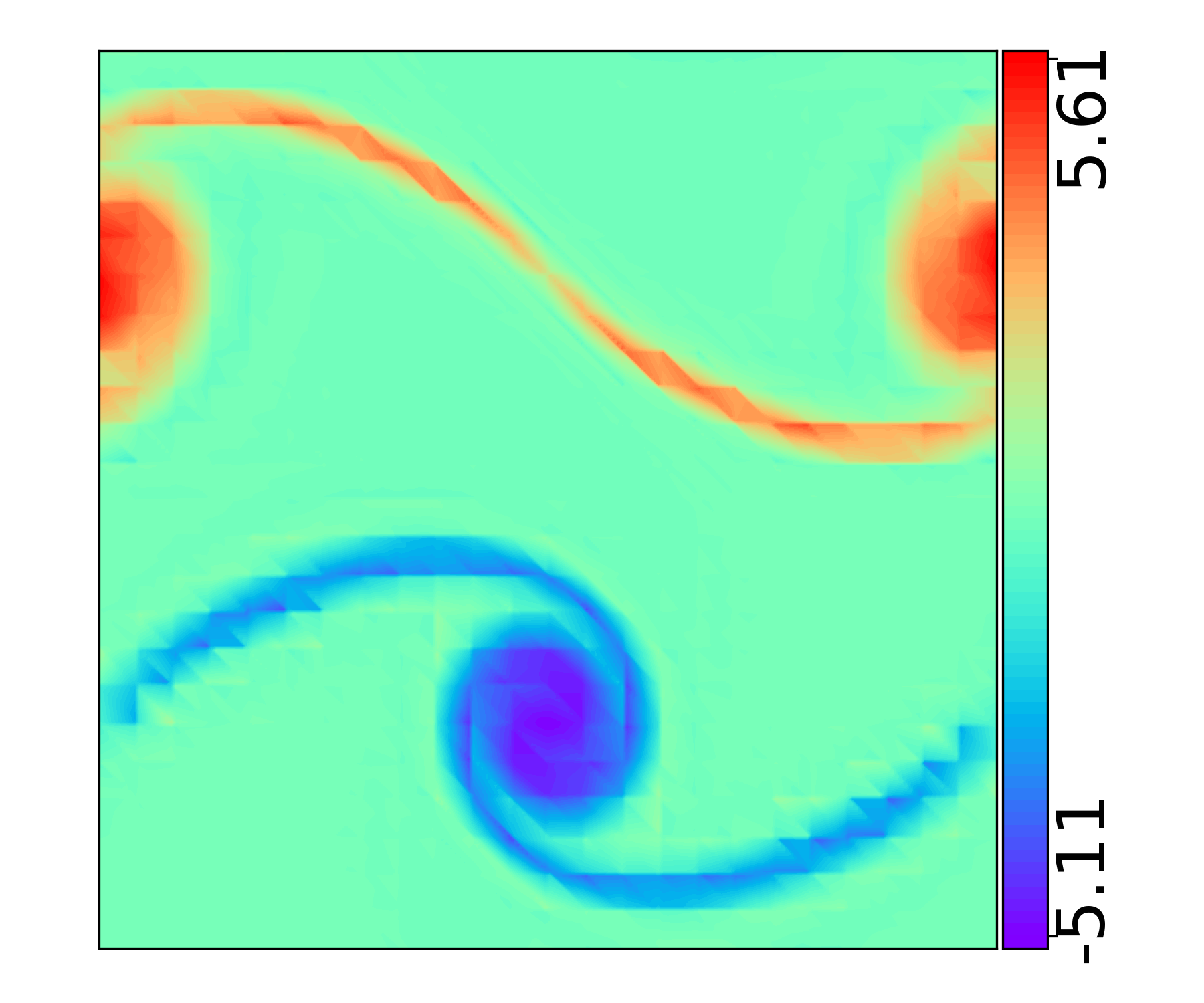}\\
    \includegraphics[width=0.425\linewidth]{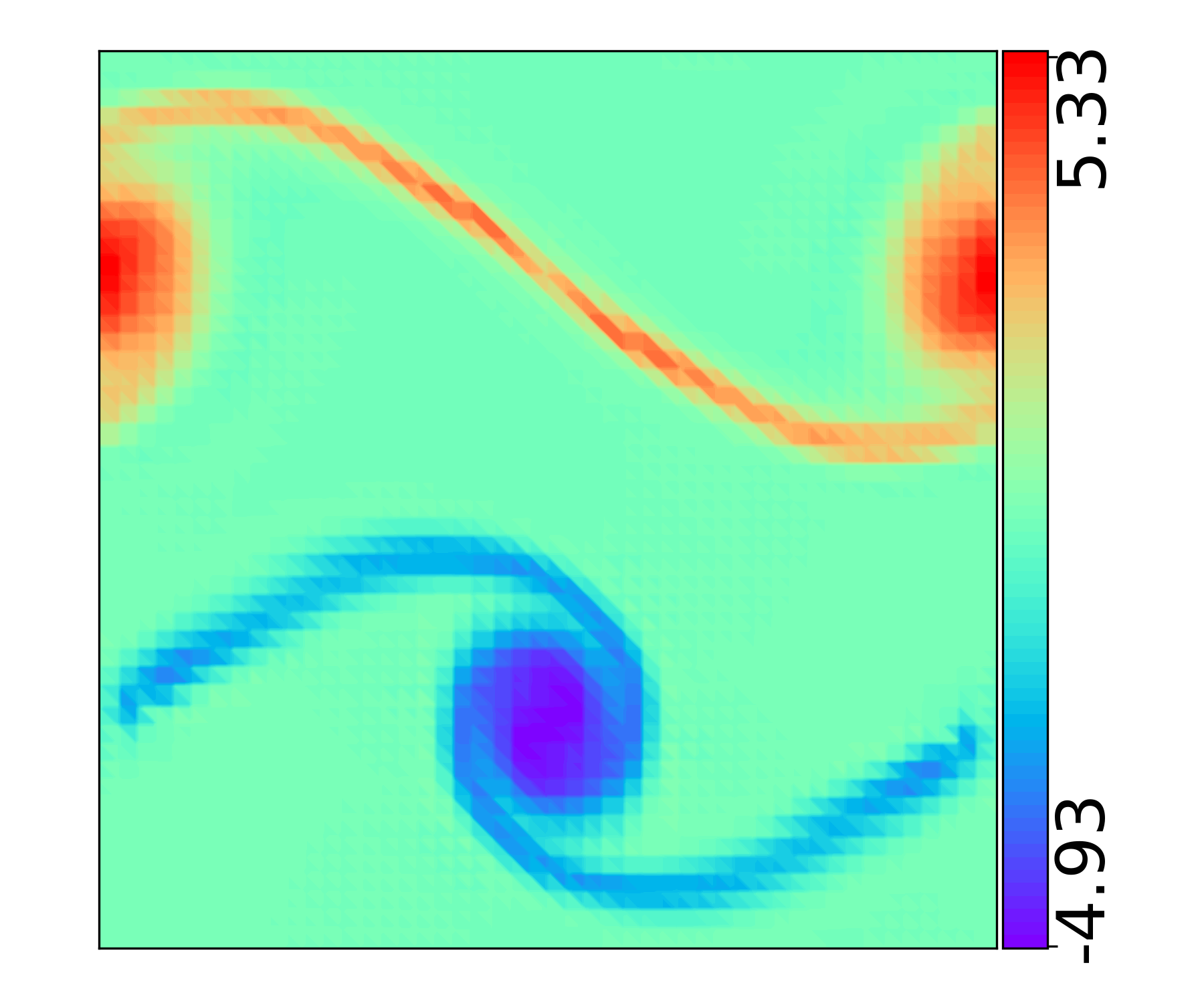}
    \includegraphics[width=0.425\linewidth]{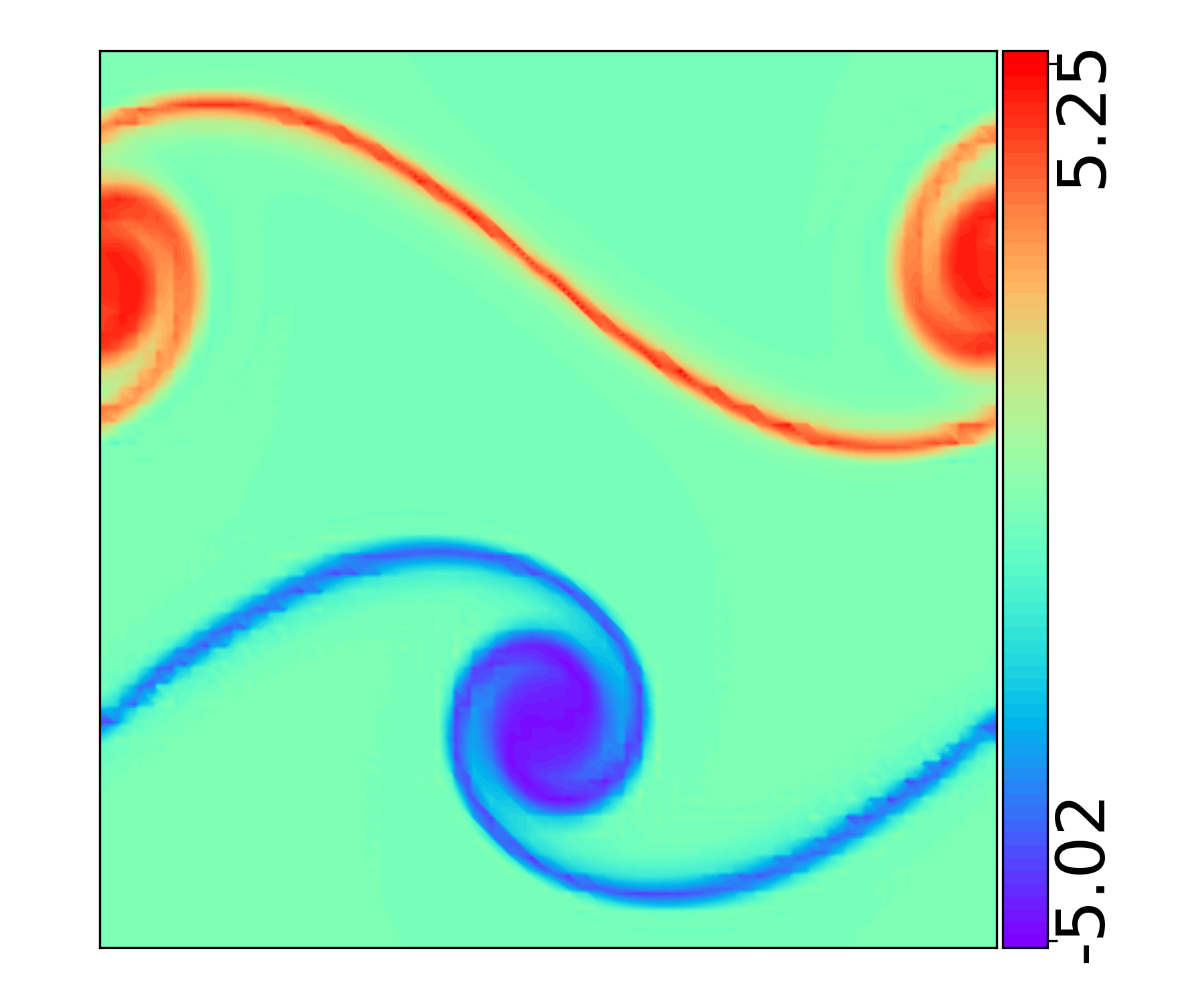}\\
    \includegraphics[width=0.425\linewidth]{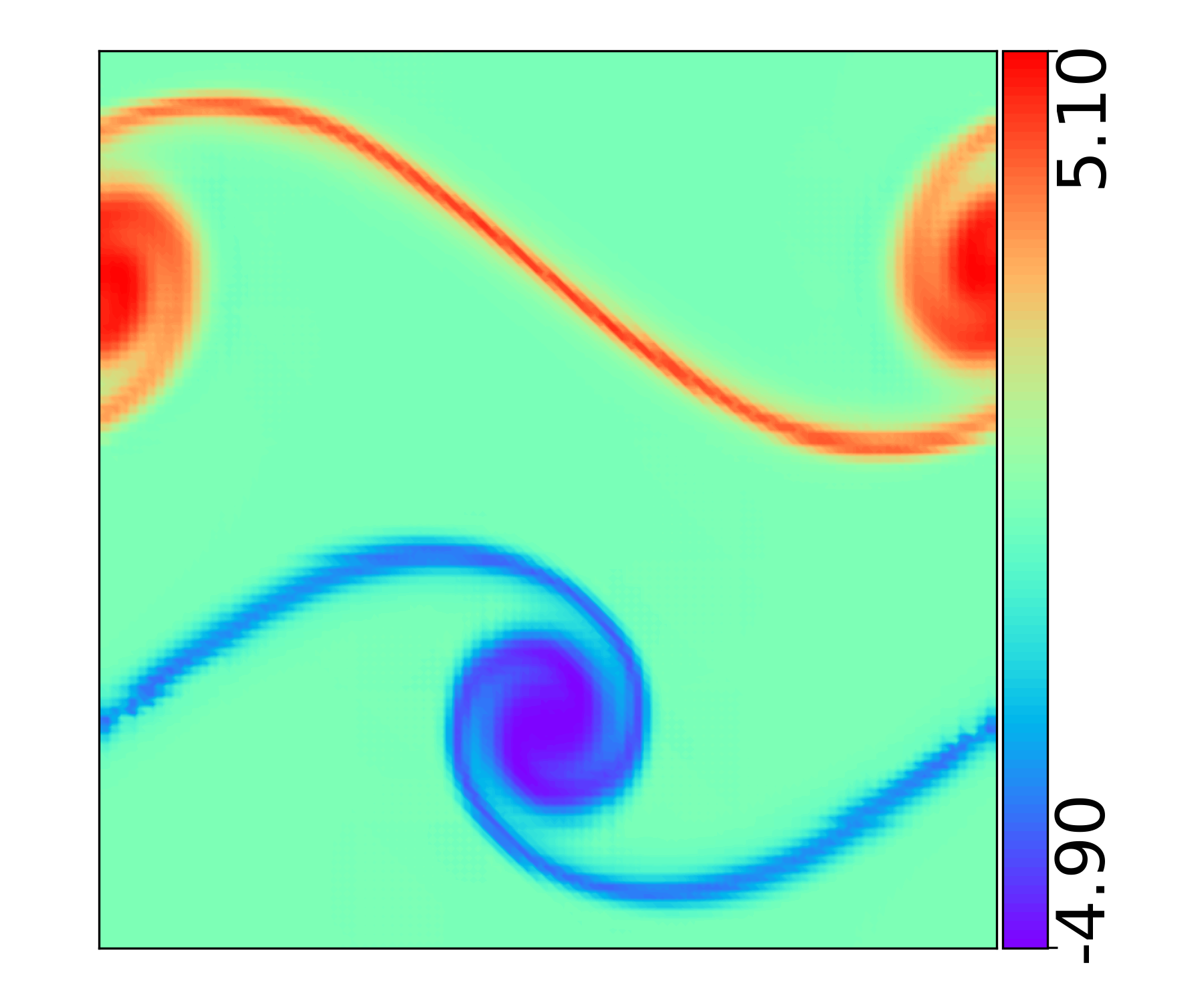}
    \includegraphics[width=0.425\linewidth]{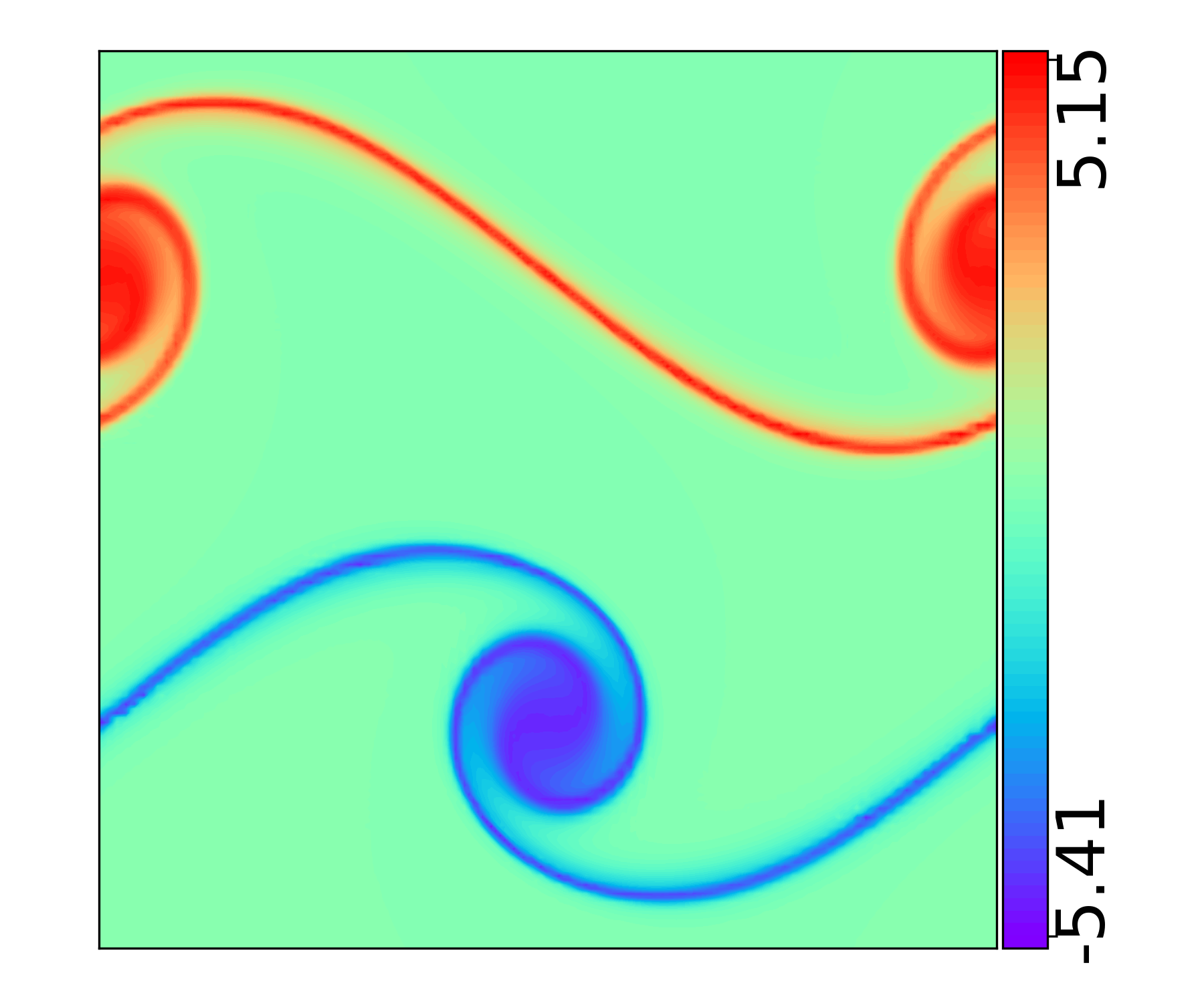}
    \caption{Experiment 2.  Vorticity contours at time $t=6$, computed by the $RT_k/P_k$ method.  From top to bottom, $h=0.7405,0.3702,0.1851,0.0926$; from left to right, $k=1,2$.}
    \label{fig:exp2-vortex}
\end{figure}

\appendix
\appendixnotitle

\subsection{Proof of Lemma \ref{lem_wsuni}}\label{ap_ws}
In this section, we prove the weak--strong uniqueness stated in Lemma \ref{lem_wsuni}. 
\begin{proof}[Proof of Lemma \ref{lem_wsuni}]
Let us recall that a DW solution $\vu$ satisfies \eqref{dws} and a strong solution $\tvu$ satisfies \eqref{dw1} without defect. 
Recall the definition of the relative energy functional
$\RE(\vu|\tvu)= \intO{\frac12 |\vu-\tvu|^2} = \intOB{\frac12 |\vu|^2 -\vu\cdot \tvu+ \frac12 |\tvu|^2}.$  Consequently, 
\begin{align*}
&[\RE (\vu|\tvu)]_0^{\tau} = 
\left[ \intOB{\frac12 |\vu|^2 -\vu \cdot \tvu+ \frac12 |\tvu|^2} \right]_0^{\tau}
 \leq   -\int_{{\Omega}} d\frakE(\tau) -    \left[ \intO{ \vu\cdot \tvu} \right]_0^{\tau}
 \\& =  -\int_{{\Omega}} d\frakE(\tau) 
- \intT{\intOB{\vu \cdot \partial_t\tvu + (\vu \otimes \vu): \Grad\tvu
    }} 
-\intT{\int_{{\Omega}} \Grad\tvu: d\frakR(t)}.
\end{align*}
As $\vu$ and $\tvu$ share the same initial data, $\RE(\vu|\tvu)(0)=0$. 
Thus,
\begin{equation}\label{RE1}
\begin{aligned}
\RE (\vu|\tvu)(\tau) + \int_{{\Omega}} d\frakE(\tau)  
\leq &
- \intT{\intOB{\vu \cdot \partial_t\tvu + (\vu \otimes \vu): \Grad\tvu
    }} 
\\&
-\intT{\int_{{\Omega}} \Grad\tvu: d\frakR(t)}.
\end{aligned}
\end{equation}
Recalling \eqref{dw2} and the fact that $\tvu$ satisfies \eqref{PDE}, we have
\begin{equation}\label{RE2}
    \intO{\vu \cdot \pdt \tvu} = 
-\intO{\vu \cdot (\tvu \cdot \Grad \tvu + \Grad \tp)} = 
-\intO{\tvu \cdot \Grad \tvu \cdot \vu} .
\end{equation}
 Substituting \eqref{RE2} into \eqref{RE1} we obtain
\begin{equation*}
\begin{aligned}
&\RE (\vu|\tvu)(\tau) + \int_{{\Omega}} d\frakE(\tau)  
\leq 
 \intT{\intO{ (\tvu-\vu)  \cdot \Grad\tvu \cdot \vu  }} 
-\intT{\int_{{\Omega}} \Grad\tvu: d\frakR(t)}
\\& = 
\intT{\intO{ (\tvu-\vu)  \cdot \Grad\tvu \cdot (\vu-\tvu)  }} 
-\intT{\int_{{\Omega}} \Grad\tvu: d\frakR(t)}
\\& \leq c \intTB{\RE (\vu|\tvu)(t)+  \int_{{\Omega}} \operatorname{tr}[d\frakR(t)]},
\end{aligned}
\end{equation*}
where the constant $c$ depends on $\norm{\tvu}_{L^\infty W^{1,\infty}}$ and we have used
\[
 \int_\Omega(\tvu-\vu)\cdot\Grad\tvu\cdot\tvu\,\dx
 =\int_\Omega(\tvu-\vu)\cdot\Grad\frac{|\tvu|^2}{2}\,\dx=\int_\Omega \frac{|\tvu|^2}{2}\; \Div (\tvu-\vu) \,\dx =0
\]
in view of the divergence-free condition for both $\vu$ and $\tvu$.
Using the compatibility condition $\operatorname{tr}\frakR\leq\overline d\,\frakE$, we obtain
\[
 \RE(\vu|\tvu)(\tau)+\int_\Omega\dd\frakE(\tau)
 \leq c\int_0^\tau\left(\RE(\vu|\tvu)(t)+\int_\Omega\dd\frakE(t)\right)\dt.
\]
Gronwall's lemma implies that both terms on the left-hand side vanish. Consequently, $\vu=\tvu$, $\frakE=0$, and $\frakR=0$ by \eqref{dw4}. The proof of Lemma~\ref{lem_wsuni} is complete.  
\end{proof}
\subsection{Proof of Theorem \ref{thm_re}}\label{ap_re}
\begin{proof}[Proof of Theorem \ref{thm_re}]
First, recalling the stability condition~\eqref{CA1}, the consistency formulation of the momentum equation~\eqref{CA3}, and the energy conservation of the strong solution, we deduce
\begin{align*}
[\RE (\vuh|\tvu)]_0^{\tau} &= 
\left[ \intOB{\frac12 |\vuh|^2 -\vuh \cdot \tvu+ \frac12 |\tvu|^2} \right]_0^{\tau}
 \leq  -\left[ \intO{ \vuh \cdot \tvu} \right]_0^{\tau}
 \\& =
- \intT{\intOB{\vuh \cdot \partial_t\tvu + (\vuh \otimes \vuh): \Grad\tvu
    }} - e_{\vm}(\tau,h,\tvu). 
\end{align*}
Next, recalling the consistency formulation of the divergence-free condition~\eqref{CA2} and using the fact that $\tvu$ satisfies~\eqref{PDE}, we obtain for a.a. $t\in(0,T)$
\begin{equation*}
- \intO{\vuh \cdot \pdt \tvu} = 
\intO{\vuh \cdot (\tvu \cdot \Grad \tvu + \Grad \tp)} = 
\intO{\tvu \cdot \Grad \tvu \cdot \vuh} +  e_{\vr}(t,h,\tp).
\end{equation*}
Combining the above two relations, we obtain
\begin{equation*}
\begin{aligned}
[&\RE (\vuh|\tvu)]_0^{\tau} 
\leq 
 \intT{\intO{ (\tvu-\vuh)  \cdot \Grad\tvu \cdot \vuh  }} 
- e_{\vm}(\tau,h,\tvu) +  \int_0^\tau e_{\vr}(t,h,\tp)\,\dt
\\& = 
\intT{\intO{ (\tvu-\vuh)  \cdot \Grad\tvu \cdot (\vuh-\tvu)  }} 
+\intT{\intO{ (\tvu-\vuh)  \cdot \Grad\frac{|\tvu|^2}{2}   }} 
\\& \quad - e_{\vm}(\tau,h,\tvu) +  \int_0^\tau e_{\vr}(t,h,\tp)\,\dt
\\& \leq c \intT{\RE (\vuh|\tvu)(t)} 
- e_{\vm}(\tau,h,\tvu)
- \int_0^\tau e_{\vr}\left(t,h,\frac{|\tvu|^2}{2}-\tp\right)\,\dt.
\end{aligned}
\end{equation*}
Here, we have used $\intO{\tvu\cdot\Grad|\tvu|^2} = -\intO{\Div \tvu |\tvu|^2}=0$ and the consistency formulation~\eqref{CA2} with the test function $|\tvu|^2/2$. The constant $c$ depends on $\norm{\tvu}_{L^\infty W^{1,\infty}}$. 
Finally, Gronwall's lemma yields
\begin{align*}
\sup_{0\leq\tau\leq T}\RE(\vuh|\tvu)(\tau)
&\aleq e^{c_1T}\bigg(\RE(\vuh|\tvu)(0)
+c_2\sup_{0\leq\tau\leq T}|e_{\vm}(\tau,h,\tvu)| \\
&\qquad\qquad
+c_2\int_0^T\left|e_{\vr}\left(t,h,\frac{|\tvu|^2}{2}-\tp\right)\right|\,\dt\bigg),
\end{align*}
where $c_1$ and $c_2$  depend on $\norm{\tvu}_{L^\infty W^{1,\infty}}$. 
This completes the proof. 
\end{proof}
\subsection{Proof of Lemma~\ref{lem_Zh}}\label{ap4}
\begin{proof}[Proof of Lemma~\ref{lem_Zh}]
Fix $K\in\grid$ and a point $\vx_K\in K$. We recall the equivalent direct-sum decomposition
\[
 RT_k(K)=P_k^d(K)\oplus(\bm x-\bm x_K)\widetilde P_k(K),
\]
where $\widetilde P_k(K)$ denotes the space of polynomials homogeneous of degree $k$ with respect to $\bm x-\bm x_K$. Thus,
\[
 \vvh=\bm p_k+(\bm x-\bm x_K)q_k,
 \qquad \bm p_k\in P_k^d(K),\quad q_k\in\widetilde P_k(K).
\]
Euler's identity for homogeneous polynomials gives
\[
 \Div\big((\bm x-\bm x_K)q_k\big)
 =d q_k+(\bm x-\bm x_K)\cdot\Grad q_k=(d+k)q_k.
\]
Since $\Div\bm p_k\in P_{k-1}(K)$, the degree-$k$ homogeneous part of $\Div\vvh$ is $(d+k)q_k$. Hence, $\Div\vvh=0$ implies $q_k=0$, and consequently $\vvh|_K=\bm p_k\in P_k^d(K)$. Differentiation yields $\Gradh\vvh|_K\in P_{k-1}^{d\times d}(K)$. For $k=0$, this means that $\vvh$ is constant on each element. This completes the proof.
\end{proof}
\subsection{Proof of Lemma~\ref{lem_exi}}\label{app_exi}
\begin{proof}[Proof of Lemma~\ref{lem_exi}]
For any fixed $h \in (0,1)$, we first consider the velocity equation restricted to $Z_h$. It can be written as the finite-dimensional ordinary differential equation
\[
    \partial_t\vu_h=\mathcal F_h(\vu_h),
\]
where $\mathcal F_h:Z_h\to Z_h$ is defined by 
\[  (\mathcal F_h(\vw_h),\vv_h)_{\mathcal T_h}
    =
    -b(\vw_h,\vw_h;\vv_h)\; 
    \forall\vv_h\in Z_h .\]
The inequality $  \big||a|-|b|\big|\leq |a-b|$
together with the polynomial inverse and trace estimates implies that $\mathcal F_h$ is locally Lipschitz on the finite-dimensional space $Z_h$. Hence, the Picard--Lindel\"of theorem yields a unique local solution.

Testing the equation by $\vu_h$ and using the energy identity gives
\[
    \frac12\frac{\mathrm d}{\mathrm dt}\|\vu_h\|_{L^2}^2
    +\seminorm{\vu_h}_{\vu_h}^2=0 .
\]
Therefore, the solution remains bounded in $Z_h$. Since the local solution of a finite-dimensional ODE can be extended as long as it remains bounded, the maximal existence time is infinite. Thus the velocity solution exists globally.

The pressure is then recovered from the discrete inf--sup condition of the Raviart--Thomas pair. This completes the proof. 
\end{proof}

\end{document}